\documentclass[letter, 11pt]{article} 
\usepackage{amsmath}
\usepackage{mathrsfs} 
\usepackage{amsthm}
\usepackage{amsfonts}
\usepackage{amssymb}
\usepackage{latexsym}
\usepackage{tikz} 
\usepackage{esint}
\usepackage{mathtools}
\usepackage{stmaryrd}

\usepackage{mathabx}

\usepackage{tikz-cd}
\usepackage{cancel}
\usepackage{sectsty}
\usepackage{hyperref}
\usepackage{caption}
\usepackage{float}
\usepackage{setspace}

\usepackage{comment}
\usepackage{quiver}

\usepackage{xcolor}

\definecolor{mediumtealblue}{rgb}{0.0, 0.33, 0.71}
\definecolor{red(munsell)}{rgb}{0.96, 0.12, 0.24}
\definecolor{forestgreen}{rgb}{0.13, 0.55, 0.13}
\definecolor{ginger}{rgb}{0.69, 0.4, 0.0}

\newcommand\myshade{85}
\colorlet{mylinkcolor}{red(munsell)}
\colorlet{mycitecolor}{mediumtealblue}
\colorlet{myurlcolor}{forestgreen}

\usepackage{lmodern}

\usepackage[utf8]{inputenc}

\usepackage[inline,shortlabels]{enumitem}

\usepackage{amssymb,mathrsfs,stmaryrd}
\usepackage{mathtools}
\usepackage{thmtools,thm-restate}
\usepackage{colonequals}
\usepackage{tikz}
\usetikzlibrary{cd}

\usepackage[capitalize,noabbrev]{cleveref}

\hypersetup{
	linkcolor  = mylinkcolor!\myshade!black,
	citecolor  = mycitecolor!\myshade!black,
	urlcolor   = myurlcolor!\myshade!black,
	colorlinks = true
}

\mathchardef\mhyphen="2D

\usepackage{titlesec}
\titleformat{\subsection}[runin]{\normalfont\bfseries}{\thesubsection}{1em}{}
\titleformat{\subsubsection}[runin]{\normalfont\bfseries}{\thesubsubsection}{1em}{}

\renewcommand{\thesubsubsection}{(\thesubsection.\arabic{subsubsection})}

\usepackage{tocloft} 

\usepackage[OT2,T1]{fontenc}
\DeclareSymbolFont{cyrletters}{OT2}{wncyr}{m}{n}
\DeclareMathSymbol{\Sha}{\mathalpha}{cyrletters}{"58}

\sectionfont{\fontsize{12}{15}\selectfont}
\subsectionfont{\fontsize{11}{15}\selectfont}
\subsubsectionfont{\fontsize{10}{15}\selectfont}

\usepackage[colorinlistoftodos,textsize=scriptsize]{todonotes}
\usepackage{marginnote}

\newcommand{\mytodo}[2][]{{%
 \let\marginpar\marginnote
 \reversemarginpar
 \renewcommand{\baselinestretch}{0.8}%
 \todo[#1]{#2}}}

\usepackage[matrix,tips,graph,curve]{xy}
\usepackage{enumitem}

\makeatletter

\usepackage[
backend=biber,
style=alphabetic,
sorting=anyt, doi=false, url=false, isbn=false, giveninits=true, maxnames = 6,
]{biblatex}

\makeatother


\makeatletter

\theoremstyle{plain}
\newtheorem{theorem}[subsubsection]{Theorem}
\newtheorem{corollary}[subsubsection]{Corollary}
\newtheorem{lemma}[subsubsection]{Lemma}
\newtheorem{proposition}[subsubsection]{Proposition}

\theoremstyle{definition}
\newtheorem{definition}[subsubsection]{Definition}
\newtheorem{remark}[subsubsection]{Remark}
\newtheorem{example}[subsubsection]{Example}

\newcommand{\IA}{\mathbb{A}}

\newcommand{\IC}{\mathbb{C}}
\newcommand{\ID}{\mathbb{D}}

\newcommand{\IF}{\mathbb{F}}
\newcommand{\IG}{\mathbb{G}}

\newcommand{\IN}{\mathbb{N}}

\newcommand{\IP}{\mathbb{P}}
\newcommand{\IQ}{\mathbb{Q}}

\newcommand{\IV}{\mathbb{V}}

\newcommand{\IZ}{\mathbb{Z}}

\newcommand{\sC}{\mathcal{C}}

\newcommand{\End}{\mathrm{End}}

\newcommand{\Hom}{\mathrm{Hom}}
\newcommand{\Aut}{\mathrm{Aut}}

\renewcommand{\deg}{\mathrm{deg}}

\newcommand{\Spec}{\mathrm{Spec}}

\newcommand{\che}[1]{\widecheck{#1}}

\newcommand\iso{\,{\cong}\,} 
\newcommand\tensor{{\otimes}}

\newcommand{\<}{\langle}
\renewcommand{\>}{\rangle}

\newcommand{\into}{\hookrightarrow}

\def\d/{/\mspace{-6.0mu}/}

\def\wt{\widetilde}

\newcommand{\frakm}{\mathfrak{m}}

\newcommand{\p}{\partial}

\newcommand{\tf}{\wt{f}}

\newcommand{\fp}{\mathfrak{p}}

\newcommand{\sV}{\mathcal{V}}
\newcommand{\sI}{\mathcal{I}}

\newcommand{\sJ}{\mathcal{J}}
\newcommand{\sF}{\mathcal{F}}
\newcommand{\sH}{\mathcal{H}}
\newcommand{\sX}{\mathcal{X}}
\newcommand{\sS}{\mathcal{S}}
\newcommand{\sT}{\mathcal{T}}
\newcommand{\sE}{\mathcal{E}}
\newcommand{\sO}{\mathcal{O}}

\newcommand{\sZ}{\mathcal{Z}}
\newcommand{\sP}{\mathcal{P}}

\newcommand{\shPic}{\mathscr{P}}
\newcommand{\Pic}{\mathrm{Pic}}
\newcommand{\Fil}{\mathrm{Fil}}
\renewcommand{\H}{\mathrm{H}}
\newcommand{\Spf}{\mathrm{Spf}}

\newcommand{\shP}{\mathscr{P}}
\newcommand{\Def}{\mathrm{Def}}

\newcommand{\dR}{\mathrm{dR}}
\newcommand{\et}{\textrm{{\'e}t}}
\newcommand{\cris}{\mathrm{cris}}

\newcommand{\Cris}{\mathrm{Cris}}
\newcommand{\bH}{\mathbf{H}}

\newcommand{\NS}{\mathrm{NS}}
\newcommand{\GL}{\mathrm{GL}}
\newcommand{\codim}{\mathrm{codim}}

\newcommand{\sfM}{\mathsf{M}}

\newcommand{\Gal}{\mathrm{Gal}}
\newcommand{\sA}{\mathcal{A}}

\newcommand{\bxi}{\boldsymbol{\xi}}

\newcommand{\fm}{\mathfrak{m}}

\renewcommand{\tf}{\mathrm{tf}}
\newcommand{\Mon}{\mathrm{Mon}}

\newcommand{\red}{\mathrm{red}}

\newcommand{\pr}{\mathrm{pr}}

\newcommand\restr[2]{{
	\left.\kern-\nulldelimiterspace
	#1
	\vphantom{\big|}
	\right|_{#2}
	}}

\newcommand{\cl}{\mathrm{cl}}

\newcommand{\mpr}{}

\title{\textbf{\large{Arithmetic Deformation of Line Bundles}}}

\author{\normalsize{David Urbanik and Ziquan Yang}}

\vspace{-1ex}
  
\date{\today}

\begin{document}

\maketitle

\begin{abstract}
    We introduce a new method to study mixed characteristic deformation of line bundles. In particular, for families $f : \mathcal{X} \to \mathcal{S}$ with big monodromy and large period image defined over the ring of $N$-integers $\mathcal{O}_{L}[1/N]$ of a number field $L$, we produce a proper closed subscheme $\mathcal{E} \subsetneq \mathcal{S}$ outside of which all line bundles appearing in positive characteristic fibres of $f$ admit characteristic zero lifts. This in particular applies to elliptic surfaces over $\mathbb{P}^1$ and projective hypersurfaces in $\mathbb{P}^3$ of degree $d \geq 5$. We also study the locus $\mathcal{E}$ in more detail in the $h^{0, 2} = 2$ case.
\end{abstract}

\tableofcontents



\section{Introduction}

\subsection{Main Results} \label{sec: main results}

In \cite{Del02}, Deligne proved that every line bundle on a K3 surface in positive characteristic lifts to characteristic $0$ along with the surface. Deligne's theorem and its extensions have historically played a pivotal role in addressing the Tate conjecture for K3 surfaces and varieties alike in positive characteristic, because a mainstream method relies on considering characteristic $0$ liftings and connecting the Tate conjecture to its Hodge-theoretic analogue, i.e., the Lefschetz $(1, 1)$-theorem (e.g., \cite{Nygaard, NO, MPTate, HYZ}). To explore the potential for future advances in this direction, we pose the following question: \textit{In an arithmetic family, can we generate all line bundles on the characteristic $p$ fibers by specializing line bundles on the characteristic $0$ fibers?}

To set up a context for the question, let $L \subseteq \IC$ be a number field and $\sO_L$ be its ring of integers. Suppose that $\sS$ is a smooth quasi-projective $\sO_L[1/N]$-scheme for some $N \in \IN_{> 0}$ with $\sS_\IC$ connected, and $f : \sX \to \sS$ is a polarized smooth projective family with geometrically connected fibers. Then we provide the following answer to address the above question. 

\begin{theorem}
\label{mainthm1}
Suppose that $f_\IC$ has large monodromy and sufficiently large period image. Then there exists a proper closed subscheme $\sE \subsetneq \sS$ such that for every algebraically closed field $k$ of positive characteristic and $s \in \sS(k) \setminus \sE(k)$, every line bundle $\xi$ on the fibre $\sX_{s}$ lifts to characteristic $0$. 
\end{theorem}

By ``$\xi$ lifts to characteristic $0$'', we mean that there exists a finite extension $K$ of $W(k)[1/p]$ and a point $\wt{s} \in \sS(\sO_K)$ lifting $s$, such that $\xi$ lies in the image of the specialization map $\Pic(\sX_{\wt{s}}) \to \Pic(\sX_s)$. We now define the conditions in the hypothesis. Their necessity will be discussed in \S\ref{sec: non-examples}. 

\begin{definition}
\label{def: large monodromy}
Let $S$ be a connected $\IC$-variety and $X \to S$ be a smooth projective morphism with connected fibers. 
\begin{itemize}
    \item We say that $X \to S$ has \emph{large monodromy} if for some $\IC$-point $s \in S$ not lying on the Noether-Lefschetz locus, the Zariski closure of the image of the monodromy representation\footnote{\mpr{When taking Betti cohomology, we take $X_s$ to mean the analytic space associated to the variety $X_s$ and do not make a notational difference.}}
    \[ \pi_{1}(S, s) \to \mathrm{O}(\mathrm{H}^{2}(X_s, \IQ)) \]
    contains $\mathrm{SO}(\mathrm{T}^{2}(X_{s}, \IQ))$ as the identity component. Here we are equipping $\H^2(X_s, \IQ)$ with the symmetric bilinear form induced by the chosen polarization, and taking $\mathrm{T}^{2}(X_{s}, \IQ)$ to be the orthogonal complement of the subspace of $(1, 1)$-classes. 
    \item We say that $X \to S$ has \emph{sufficiently large period image} if for some (and hence a general) $\IC$-point $s \in S$, the Kodaira-Spencer map 
    \begin{equation}
        \label{eqn: Kodaira-Spencer}
        T_{s} S \to \Hom(\H^1(\Omega^1_{X_s}), \H^2(\sO_{X_s}))
    \end{equation}
    has image of dimension at least $h^{0, 2}(X_s) = \dim_{\IC} \H^{2}(\sO_{X_s})$.
\end{itemize}
\end{definition}

The proof of Thm~\ref{mainthm1} yields structural theorems for the \textit{arithmetic Noether-Lefschetz (NL) locus,} and the relative Picard scheme $\shP_{\sX/\sS}$. We say that a proper closed reduced subscheme $Z \subseteq \sS$ is a \textit{component of the arithmetic NL locus} if it is the scheme-theoretic image of an irreducible component of the reduced subscheme of $\shP_{\sX/\sS}$ which does not dominate $\sS$. Then the arithmetic NL locus, which we shall denote by $\mathcal{N}$, is the collection of all such (typically infinitely many) components. It is characterized by the property that a geometric point $s \to \sS$ factors through $\mathcal{N}$ if and only if $\mathrm{rank\,} \NS(\sX_{s}) > \mathrm{rank\,} \NS(\sX_{\bar{\eta}})$, where $\bar{\eta}$ is a geometric generic point of $\sS$. Then we have the following theorem. 

\begin{theorem}
\label{mainthm2}
   Under the hypothesis of Thm \ref{mainthm1}, there exists a closed proper subscheme $\sE \subsetneq \sS$ such that the restriction of $\shP_{\sX/\sS}$ over $\sS \smallsetminus \sE$ is syntomic (i.e., flat and lci) over $\sO_L[1/N]$, and the restriction of every component of $\mathcal{N}$ over $\sS \smallsetminus \sE$ is flat over $\sO_L[1/N]$ and has codimension exactly $h^{0, 2}(\sX_{\bar{\eta}})$ in $\sS \smallsetminus \sE$.
 \end{theorem}
In particular, the above theorem implies that when restricted to $\sS \smallsetminus \sE$ the arithmetic NL locus $\mathcal{N}$ can be recovered from the variation of Hodge structures (VHS) on $R^2 f_{\IC *} \IQ$, as every component of $\mathcal{N}$ is the Zariski closure of some component of the usual (Hodge-theoretic) NL locus over $\sS(\IC)$. We also remark that our method will give an explicit construction of $\sE$: Away from small primes, $\sE$ is the Zariski closure of a so-called ``atypical Hodge locus'' on $\sS_\IC$ (\cite{BKU}) and the locus where the Kodaira-Spencer map has rank $< h^{0, 2}$. The issue of determining $\sE$ effectively is discussed at the end of \S\ref{sec: methods}. 

To illustrate what might happen on $\sE$, we give a thorough analysis in the case when $h^{0, 2} = 2$. See \S\ref{sec: non-liftable}. One conclusion we can draw is the following (cf. \ref{thm: h = 2 case}): 

\begin{theorem}
\label{thm: h = 2 case rough}
    In the situation of \ref{mainthm1}, suppose in addition for every $s \in \sS(\IC)$, $h^{0, 2}(\sX_s) = 2$ and the Kodaira-Spencer map (\ref{eqn: Kodaira-Spencer}) has rank $\ge 2$. 

    Then for $p \gg 0$, the following holds: Let $k$ be an algebraically closed field of characteristic $p$ and take any $s \in S(k)$ and $\xi \in \Pic(\sX_s)$. Up to equi-characteristic deformations of the pair $(s, \xi)$, $m\xi$ is a sum of liftable line bundles for some $m \in \IN_{> 0}$. 
\end{theorem}

The upshot of the above conclusion is that, when $p \gg 0$ and $s$ lies on $\sE$, even if $\xi$ is non-liftable, it still comes from specializing and deforming line bundles in characteristic $0$, albeit in a more indirect way. However, in general we expect that the complexity of possible behaviours that occur over the locus $\sE$ to grow quickly as $h^{0, 2}$ increases. 




\subsection{Applications} We will give applications of our results to general type and elliptic surfaces. Let $k$ be an algebraically closed field of characteristic $p$. 


\begin{theorem}
\label{thm: surfaces in P3 intro}
    Let $d \ge 5$ be an integer. If $p \gg 0$, then for a general degree $d$ hypersurface $Y$ in $\IP^3_k$, every $\xi \in \Pic(Y)$ is liftable. 
\end{theorem}
\noindent Here the word ``general'' means that $Y$ comes from a Zariski open dense subscheme of a natural parameter space. Note that the above surfaces are of general type.

Let $\pi : X \to \IP^1$ be an elliptic surface (with a zero section) of height $h := \mathrm{deg\,}(R^1 \pi_* \sO_X)^\vee$. We say that $X$ is \textit{admissible} if all geometric fibers of $\pi$ are irreducible, or equivalently the Weierstrass normal form of $X$ is smooth. This condition is satisfied on a Zariski dense open subscheme of the coarse moduli space of elliptic surfaces over $\IP^1$ with a fixed $h$. 


\begin{theorem}
\label{thm: elliptic in intro}
    If $p \gg 0$, then for a general admissible $X$ over $k$, every $\xi \in \Pic(X)$ is liftable.
\end{theorem}

\noindent When $h = 3$, \ref{thm: h = 2 case rough} allows us to remove the word ``general'' at the expense of a weaker conclusion: 

\begin{theorem}
    If $p \gg 0$, $h = 3$ and $X$ is admissible, then for every $\xi \in \Pic(X)$, we may equi-characteristically deform $(X, \xi)$ to another pair $(X', \xi')$, such that some positive multiple of $\xi'$ is a sum of liftable line bundles.
\end{theorem}

The precise forms of the above theorems are given in \S\ref{sec: examples}.  

Finally we remark that elliptic surfaces are of particular interest because their line bundles give rise to rational points on their generic fibers, which are elliptic curves over function fields. Hence lifting results in the style of \ref{thm: elliptic in intro} suggest a way of constructing all rational points on an elliptic curve over a global function field by eventually appealing to the Lefschetz $(1, 1)$-theorem. This idea was indeed used before, in an earlier work of the second author \cite{HYZ} joint with Hamacher and Zhao, to establish a new case of the function field BSD conjecture. It is our hope that \ref{thm: elliptic in intro} is a start for future progress.





\subsection{Methods}
\label{sec: methods}
The key new idea in our work is that one can apply the theory of Hodge-theoretic unlikely intersections over $\IC$ to understand the liftability of line bundles in positive characteristic. This is noteworthy because we are using highly transcendental tools to answer a completely algebraic question. To explain how this works, let us first give an informal overview of how line bundles deform, starting with the $\IC$-fibre $\sS_{\IC}$. If $Z \subset \sS_{\IC}$ is a component of the Noether-Lefschetz locus coming from deforming a line bundle $\xi$ on $\sX_s$ for some $s \in \sS(\IC)$, then analytically locally at $s$ one can view $Z$ as the set of points $t$ where parallel transport of the Betti class $c_1(\xi)$ to $t$ remains in $\Fil^{1} \H^2_\dR(\sX_{t}/\IC)$, the middle piece of the Hodge filtration. Then it is not hard to see that $Z$ is locally cut out at $s$ by $h^{0, 2}(\sX_s)$ equations.

Now, by generalizing a result of Deligne \cite[Prop.~1.5]{Del02}, we explain in \ref{thm: Deligne generalized} that under some mild assumptions the same statement also applies on the integral level. Suppose that $k$ is an algebraically closed field of characteristic $p > 0$, $s$ is a $k$-point on $\sS$ and $\xi \in \Pic(\sX_s)$, and assume for simplicity that $\sO_L[1/N]$ is unramified over $\IZ[1/N]$. Then the maximal locus $\Def(\xi)$ inside the formal completion of $\sS \tensor_{\sO_L[1/N]} W(k)$ at $s$ to which $\xi$ deforms is cut out by $h^{0, 2}(\sX_s)$ equations. The analogue of the locus $Z$ in the preceding paragraph is the Zariski closure of $\Def(\xi)$ in $\sS$, which we shall denote by $\mathcal{Z}$. But --- and this where the connection to unlikely intersections appears --- passing to the integral level gains us a dimension. More specifically, if $\xi$ does not deform to characteristic zero, which is to say that $\Def(\xi)$ has no characteristic zero fiber, then the codimension of $\sZ_{k}$ inside $\sS_{k}$ must be \emph{at most} $h^{0, 2}(\sX_s) - 1$.

Let us first explain why it would be unusual to see such a phenomenon in characteristic zero. Indeed, if we consider again the locus $Z \subset \sS_{\IC}$ in the first paragraph, then we can alternatively define the germ of $Z$ at $s$ in the following way: consider a Hodge-theoretic period map $\psi : B \to \che{D}$ on a small ball $B \subset \sS(\IC)$ around $s$, and define $Z \cap B$ as $\psi^{-1}(\che{D}_{c_1(\xi)})$, where $\che{D}_{c_1(\xi)} \subset \che{D}$ is a suitable flag subvariety of codimension $h^{0, 2}$ defined by $c_1(\xi)$. Then if $\Gamma_{\psi} \subset \sS_{\IC} \times \che{D}$ is the graph of $\psi$, the intersection between $\Gamma_{\psi}$ and $\sS_{\IC} \times \che{D}_{c^{1}(\xi_{s})}$ is now unlikely: it has codimension at most $\dim \che{D} + (h^{0,2} - 1)$, smaller than the ``expected'' codimension of $\dim \che{D} + h^{0,2}$. 

Let us call such a $Z$ giving rise to an unlikely intersection of this type an ``atypical Hodge locus''.\footnote{Strictly speaking, to be in agreement with \cite{BKU}, $\che{D}$ should be an orbit of the generic Mumford-Tate group. But under the large monodromy assumption we may simply take $\che{D}$ to be the variety of all polarized Hodge flags. } Hodge loci which are atypical in this sense are studied in the recent paper \cite{BKU}, where the authors prove the following ``Geometric Zilber-Pink'' statement: the atypical Hodge locus of positive period dimension associated to an integral variation of Hodge structures $\IV$ on $S$ with $\IQ$-simple monodromy lies in finitely many subvarieties strictly contained in $S$. Thus the phenomenon just described can only occur for $Z$ lying inside one of these subvarieties; it in particular does not occur for $Z$ intersecting a Zariski open subset of $S$.

\vspace{0.5em}

At this point one can observe that what is needed for our main theorems is a kind of analogue of this ``Geometric Zilber-Pink'' result at the integral level. Tools which will make this possible were developed in recent work \cite{urbanik2022algebraic} by the first author. More specifically, the paper \cite{urbanik2022algebraic} gives a purely algebro-geometric way of obtaining the results of \cite{BKU} which also allows one to study ``integral'' unlikely intersections. The idea is that one studies infinitesimal period maps --- formal objects giving maps $\psi: \ID_r \to \che{D}$ where $\ID_r \subset \sS$ is an infinitesimal disk of order $r$ --- and looks at the locus $\sA_{r}$ in a jet space $J^d_r \sS$ of disks $\ID^d_r \subset \sS$ with ``dimension'' $d = \dim \sS_{\IC} - (h^{0, 2} - 1)$ lying in the inverse image under some such $\psi$ of carefully chosen flag subschemes.\footnote{In the proof of \ref{Econstrprop} where this argument takes place we will find it convenient not to construct $\sA_r$ directly, but instead understand it through a larger locus $\sI_r$ lying in the product of a bundle over $J^d_r \sS$ with the jet space of a flag variety.} Similar, but more algebro-geometric, reasoning to \cite{BKU} --- which ultimately reduces as in \cite{BKU} to an application of the Ax-Schanuel Theorem \cite{zbMATH07066495} --- shows that $\sA_{r,\IC}$ projects to a proper closed subvariety of $\sS_{\IC}$ for $r \gg 0$. But then the Noetherianity of $\sA_{r}$ implies the same fact at the integral level, and shows that ``unlikely intersections'' of positive dimension are constrained by intersections of the same type in characteristic zero for $p \gg 0$. 

\vspace{0.5em}

Let us make two further comments about our large monodromy and period image hypotheses. The methods of this paper should apply even when $f$ does not have large monodromy, although the reasoning becomes more complicated. In order to apply the Ax-Schanuel theorem, the flag variety $\che{D}$ must be taken to be an orbit of the algebraic monodromy group. If this group shrinks, so does $\che{D}$, which means that the expected codimension of $Z$ becomes smaller. One then has to take this into account at the crystalline level, which to our knowledge requires knowing that the global cycles defining the monodromy group satisfy absolute-Hodge-like properties. We plan future work to deal with this more refined situation. 

Secondly, the ``sufficiently large period image'' assumption is necessary to ensure that if $Z$ is atypical, it is also positive-dimensional (in the sense of period dimension). As we explain in \S\ref{sec: non-examples} below, our results are false without this assumption, and the fact that geometric Zilber-Pink reasoning requires the intersections to be positive dimensional allows one to see the necessity of this assumption in a different way. 

\vspace{0.5em}

\paragraph{Comment on Effectivity} The locus $\sE$ appearing in the statements of the theorems and various loci which appeared in the sketch above can in principle be computed by an algorithm run on a sufficiently powerful computer, given explicit projective equations for the family as input. In the case of the generic fibres of these loci (i.e., if we were to compute them over a characteristic zero field), the necessary theory for writing down these algorithms is developed in \cite{zbMATH07673321}. In particular, \cite{zbMATH07673321} gives explicit algorithms for computing the Hodge bundle, Hodge filtration, and Gauss-Manin connection from a set of projective equations as well as for computing all of the jet-related constructions we will use. Subject to limitations in the computational literature concerning algebro-geometric constructions over a more general base ring $R = \mathcal{O}_{L}[1/N]$ --- for instance, we are not aware of a detailed algorithmic approach to Chevalley's Theorem on images of constructible sets in this context --- this could be generalized to compute the locus in our paper as well.

\subsection{Notations and Conventions}
\label{sec: conventions}
(i) Let $S$ be a scheme. For a point or geometric point $s$ on $S$, we write $k(s)$ for the field of definition for $s$. An alg-geometric point on $S$ is a geometric point $s \to S$ such that $k(s)$ is algebraically (as opposed to just separably) closed. For a geometric point $s \to S$, we denote its image on the underlying topological space of $S$ by $s^\flat$. We will often consider smooth projective morphisms $\sX \to S$ such that $h^{i, j}(\sX_s) := \dim_{k(s)} \H^j(\Omega^i_{\sX_s})$ is independent of $s \in S$. In this case, we will simply write this number as $h^{i, j}$ when the family $\sX/S$ is understood. 

(ii) Suppose that $R' \to R$ is a morphism between complete local rings. Then we define the \textit{scheme theoretic image} of the induced morphism $\Spf(R) \to \Spf(R')$ to be the formal subscheme $\Spf(R'') \subseteq \Spf(R')$, where $R'' = R'/(\ker(R' \to R))$. Note that $\Spec(R'') \subseteq \Spec(R')$ is the scheme theoretic image of $\Spec(R) \to \Spec(R')$ in the usual sense \cite[01R5]{stacks-project}. We abbreviate ``local complete intersection'' to ``lci''. 

(iii) When $S$ is a smooth complex variety, by a variation of Hodge structures $\IV$ on $S$ we shall mean a $\IQ$-VHS unless otherwise noted. Moreover, we write the $\IQ$-local system as $\IV_B$ and the filtered flat vector bundle as $\IV_\dR$, so that $\IV$ is given by the pair $(\IV_B, \IV_\dR)$. Let $T$ be a connected smooth $\IC$-variety and $t \in T$ be a point. If $\IV_B$ is a $\IQ$-local system, let $\mathrm{Mon}(\IV_B, t)$ be the Zariski closure of $\pi_1(T, t)$ in $\GL(\IV_{B, t})$ and let $\mathrm{Mon}^\circ(\IV_B, t)$ denote the identity component. We call the isomorphism class of $\mathrm{Mon}^\circ(\IV_B, t)$, which is independent of the choice of $t$, the \textit{algebraic monodromy group.} When $T$ is irreducible but not necessarily smooth, the algebraic monodromy group is understood to be that of $\IV_B$ restricted to the smooth locus of $T$. Similarly, if $T$ is a connected normal Noetherian scheme and $t$ is a geometric point, and $\sH_\ell$ is an $\IQ_\ell$-local system for a prime $\ell$ invertible in $\sO_T$, define $\Mon(\sH_\ell, t)$ to be the Zariski closure of $\pi_1^\et(T, t)$ in $\GL(\sH_{\ell, t})$ and let $\Mon^\circ(\sH_\ell, t)$ be its identity component.

(iv) When $S$ is a scheme and $\sX \to S$ is a smooth projective morphism with geometrically connected fibers, we use $\shP_{\sX/S}$ to denote the relative Picard scheme (see \S\ref{sec: generalities} below). If $k$ is a field and $X$ is a geometrically connected smooth projective variety over $k$, then we let $\shP^0_{X/k}$ denote the identity component of $\shP_{X/k}$. The N\'eron-Severi group $\NS(X)$ of $X/k$ is defined to be the group of connected components $\shP_{X/k} / \shP^0_{X/k}$. For two line bundles $\xi, \xi' \in \Pic(X)$, we say that they are algebraically equivalent (denoted by $\xi \sim \xi'$) if they give the same class in $\NS(X)$; similarly, we write $\xi \sim_\IQ \xi'$ if their classes in $\NS(X)_\IQ$ differ by a $\IQ^\times$-multiple.

\section{Deformation Theory}

\subsection{Generalities on Picard Schemes}
\label{sec: generalities}
We briefly recall some basic facts about the (relative) Picard scheme and set up some preliminary lemmas. \textit{In this section, we shall always assume that $S$ is a base scheme and $f : \sX \to S$ is a smooth and projective morphism with connected geometric fibers.}

By a fundamental theorem of Artin and Grothendieck (\cite[\S8.4~Thm~3]{BLR}, see also \cite[Thm~9.4.8]{Kleiman}), the sheafification of the presheaf on the fppf site of $S$ defined by 
$$ T \mapsto \Pic(\sX_T)/\Pic(T) $$
is representable by a separated scheme $\shP_{\sX/S}$ over $S$ which is locally of finite type; moreover, every closed subscheme of $\shP_{\sX/S}$ which is of finite type over $S$ is projective over $S$. We often implicitly use the fact that the formation of $\shP_{\sX/S}$ commutes with base change (\cite[Ex.~9.4.4]{Kleiman}). Another handy fact we shall often use is that if $f$ admits a section, then $\shPic_{\sX/S}(S) = \Pic(\sX)/ \Pic(S)$ (\cite[\S8.1~Prop.~4]{BLR}), i.e., sheafification does not add new global sections.

The following fact is a special case of \cite[Prop.~9.5.20]{Kleiman}. 
\begin{proposition}
\label{prop: Kleiman}
    Assume that for every geometric point $s \to S$, the identity component $\shP^0_{\sX_s/s}$ of $\shP_{\sX_s/s}$ is smooth of a fixed dimension. Then $\shP_{\sX/S}$ has an open and closed subscheme $\shP^0_{\sX/S}$ which is of finite type over $S$ and whose fiber above a geometric point $s \to S$ is $\shP^0_{\sX_s/s}$. Moreover, if $S$ is reduced, then $\shP^0_{\sX/S}$ is smooth over $S$. 
\end{proposition} 

\begin{definition}
    We will often consider pairs of the form $(s, \xi)$ where $s \to S$ is a alg-geometric point and $\xi \in \Pic(\sX_s)$ is a line bundle on $\sX_s$. We refer to such a pair as a \textbf{LB-pair}. Note that a LB-pair corresponds bijectively to an alg-geometric point on $\shP_{\sX/S}$.
\end{definition} 



\begin{definition}
\label{def: localizations}
    Let $\shP(S,s,\xi)$ be the localization of $\shP_{\sX/S}$ at $(s, \xi)^\flat$ (i.e., the spectrum of the local ring of $\shP_{\sX/S}$ at $(s, \xi)^\flat$). We say that $(s, \xi)$ is \textbf{liftable} (resp. \textbf{strongly liftable}) if $\shP(S, s, \xi)$ has a non-empty (resp. topologically dense) characteristic $0$ fiber. When $S$ is an integral scheme, we say that $(s, \xi)$ is \textbf{global} (on $S$) if $\shP(S, s, \xi)$ dominates $S$. 
\end{definition}

Below we give a few lemmas about the above definition. 

\begin{lemma}
\label{lem: Zariski closure is flat}
    Let $(V, (\pi))$ be a DVR with fraction field $F$ and $S$ be a $V$-scheme. Let $Y$ be an $F$-scheme admitting a quasi-compact $V$-morphism to $S$. Then the scheme-theoretic image $Z$ of $Y$ is flat over $V$. 
\end{lemma}

\begin{proof}
    We may assume that $S = \Spec(A)$ and $Y = \Spec(B)$ are affine. Suppose that $Y \to S$ is given by a morphism of $V$-algebras $\varphi: A \to B$. Then $Z$ is given by $A / \ker(\varphi)$ (\cite[01R8]{stacks-project}). The flatness of $Z$ is equivalent to the $\pi$-torsion-freeness of $A / \ker(\varphi)$, i.e., if $a \not\in \ker(\varphi)$, then $\pi a \not\in \ker(\varphi)$. But this is clear because $\varphi(\pi)$ is a unit in $B$ by assumption. 
\end{proof}

\begin{lemma}
\label{lem: Deligne's little lemma}
 Let $k$ be a perfect field of characteristic $p > 0$. Set $W := W(k)$ and $K_0 = W[1/p]$. Suppose that $R$ is a Noetherian complete local ring which contains $W$ as a coefficient ring. If $R$ is flat over $W$, then $R$ admits a local $W$-morphism $R \to \sO_K$ for some finite extension $K$ of $K_0$. 
\end{lemma}

\begin{proof}
    We can adapt an argument from the proof of \cite[Thm~1.6]{Del02}. As $R$ is flat over $W$, $p$ is not a zero divisor in $R$. Let $d = \dim R - 1$ ($\dim$ denotes Krull dimension). By \cite[0BWU]{stacks-project}, there exists $x_1, \cdots, x_d$ such that $(p, x_1, \cdots, x_d)$ is a system of parameters for $R$. Consider the $W$-morphism $R_0 := W[\![t_1, \cdots, t_d]\!] \to R$ which sends $t_i$ to $x_i$. Then \cite[Lem.~16]{Cohen}\footnote{In \textit{loc. cit.}, the formal power series ring is denoted with $\{ - \}$ as opposed to $[\![-]\!]$.} tells us that $R$ is finite over $R_0$. Set $R' = R / (x_1, \cdots, x_d)$. Then the morphism $W \to R'$ is finite, as it is the base change of $R_0 \to R$ along the morphism $R_0 \to W$ which sends each $t_i$ to $0$. Since $\dim R' = 1$ by \cite[02IE]{stacks-project}, $\Spec(R')$ must have non-empty generic fiber over $\Spec(W)$. Therefore, there exists a morphism $R' \to K$ for some finite extension $K$ of $K_0$, which necessarily factors through $\sO_K$ as $R'$ is finite over $W$. 
\end{proof}

\begin{proposition}
\label{prop: liftability}
    Let $k$ be an algebraically closed field of characteristic $p > 0$ and suppose that $S$ is a scheme of finite type over $W := W(k)$. A LB-pair $(s \in S(k), \xi \in \Pic(\sX_s))$ is liftable if and only if there exists a finite extension $K$ of $K_0 := W[1/p]$ such that there exists a morphism $\wt{s} : \Spec(\sO_K) \to S$ extending $s$ and a line bundle $\wt{\xi}$ on $\sX_{\wt{s}}$ whose special fiber is $\xi$.
\end{proposition} 
\begin{proof}
The ``if'' part follows from definition: Up to replacing $K$ by a finite extension, $\wt{s}$ lifts to a morphism $\Spec(\sO_K) \to \shP_{\sX/S}$ which maps the special point to $(s, \xi)$.

Conversely, if $\shP(S, s, \xi)$ contains a characteristic $0$ point, then by taking its scheme-theoretic image we obtain a closed subscheme $B \subseteq \shP_{\sX/S}$ which contains the closed point $(s, \xi)$ and is flat over $W$ by \ref{lem: Zariski closure is flat}. By \ref{lem: Deligne's little lemma}, the formal completion of $B$ at $s$ admits some $\sO_K$-point which lifts $(s, \xi)$ for some finite extension $K/K_0$. This translates to the desired statement.  
\end{proof}

\begin{remark}
    \begin{enumerate}[label=\upshape{(\alph*)}]
        \item Roughly speaking, being ``strongly liftable'' means a LB-pair $(s, \xi)$ remains liftable after small pertubation: there exists an open neighborhood $U$ of $(s, \xi)^\flat$ in $\shP_{\sX/S}$ such that every LB-pair $(t, \zeta)$ with $(t, \zeta)^\flat \in U$ is liftable. Clearly, if $\shP(S, s, \xi)$ is flat over $\IZ_{(p)}$ for $p = \mathrm{char\,}k(s)$, then $(s, \xi)$ is strongly liftable. 
        \item Being liftable or strongly liftable is a closed property: Suppose that $(s, \xi)$ and $(t, \zeta)$ are two LB-pairs and $(t, \zeta)^\flat$ specializes to $(s, \xi)^\flat$ on $\shP_{\sX/S}$. If $(t, \zeta)$ is (strongly) liftable, then so is $(s, \xi)$. 
    \end{enumerate}
\end{remark}

\subsubsection{} \label{sec: specialization} Below we shall assume that \textit{$S$ is a notherian integral scheme with generic point $\eta$}, and give some remarks on specializing line bundles. Let $\bar{\eta}$ be a geometric point over $\eta$ and $s$ be any other geometric point. Then we can choose a morphism $\gamma : \Spec(R) \to S$ where $R$ is a strictly Henselian DVR, such that $\gamma$ sends its generic point $\eta_0$ to $\eta$ and special point $s_0$ to $s^\flat \in S$. Let $\bar{\eta}_0$ and $\bar{s}_0$ be geometric points over $\eta_0$ and $s_0$ such that their morphisms to $S$ factor through $\bar{\eta}$ and $s$ respectively. By \cite[Prop.~3.1]{MPJumping}, there are natural identifications $\NS(\sX_{\bar{\eta}}) = \NS(\sX_{\bar{\eta}_0})$ and $\NS(\sX_{\bar{s}}) = \NS(\sX_{\bar{s}_0})$. This induces a specialization map $\mathrm{sp}_\gamma : \NS(\sX_{\bar{\eta}}) \to \NS(\sX_{\bar{s}})$ (cf. \cite[Prop.~3.6]{MPJumping} and its proof). Let us also write $\gamma$ for the triple $(\gamma, \bar{\eta}_0, s_0)$.  We call such a triple \textit{a path} from $\bar{\eta}$ to $s$ as it induces an \'etale path on $S$ from $\bar{\eta}$ to $s$ in the sense of \cite[Tag~03VD]{stacks-project}, which we also denote as $\gamma$. Using \cite[054F]{stacks-project}, one deduces that the LB-pair $(s, \xi)$ is global on $S$ if and only if $\xi$ lies in the image of $\mathrm{sp}_\gamma$ for some choice of $\gamma$. 

\begin{lemma}
    \label{lem: generic LB}
    If $S$ is in addition normal, then we have the following. 
    \begin{enumerate}[label=\upshape{(\alph*)}]
        \item The natural sequence $0 \to \Pic(S) \to \Pic(\sX) \to \Pic(\sX_\eta) \to 0$ is exact. 
        \item The image of $\mathrm{sp}_\gamma \tensor \IQ$ is independent of the choice of $\gamma$. 
        \item Suppose either $\mathrm{char\,} k = 0$ or $\mathrm{char\,} k = p > 0$ but $\NS(\sX_{\bar{\eta}})$ has no $p$-torsion. Then (b) is true for the image of $\mathrm{sp}_\gamma$ integrally. 
        \item If $(s, \xi)$ a global LB-pair on $S$ and $\mathrm{Mon}(R^2 f_{\et*} \IQ_\ell, s)$ is connected for some prime $\ell$ invertible in $\sO_S$, then there exists a relative line bundle $\bxi \in \Pic(\sX)$ such that $\bxi_s \sim_\IQ \xi$. 
    \end{enumerate} 
\end{lemma} 
We remark that by (a), for any field-valued point $t$ on $S$, there are well defined specialization maps $\Pic(\sX_\eta) \to \Pic(\sX_t)$ and $\NS(\sX_\eta) \to \NS(\sX_t)$.
\begin{proof}
    (a) follows from \cite[$\mathrm{Err}_{\mathrm{IV}}$ Cor. 21.4.13]{EGAIV4}. For (b) and (c), we only explain the $\mathrm{char\,} k = p > 0$ case as the other case is similar but easier. Note that the following diagram commutes. 
        \[\begin{tikzcd}
	{\NS(\sX_{\bar{\eta}})} & {\NS(\sX_{\bar{s}})} \\
	{\H^2_\et(\sX_{\bar{\eta}}, \hat{\IZ}^p(1))} & {\H^2_\et(\sX_{\bar{s}}, \hat{\IZ}^p(1))}
	\arrow["{\mathrm{sp}_\gamma}", from=1-1, to=1-2]
	\arrow["{c_1}"', from=1-1, to=2-1]
	\arrow["{c_1}", from=1-2, to=2-2]
	\arrow["\gamma"', from=2-1, to=2-2]
    \end{tikzcd}\]
   Now if $\gamma'$ is a different choice, then the composition of \'etale paths $\gamma' \circ \gamma^{-1}$ defines an element of $\pi_1^\et(S, \bar{\eta})$, whose action on $\H^2_\et(\sX_{\bar{\eta}}, \hat{\IZ}^p(1))$ preserves the image of $\NS(\sX_{\bar{\eta}})$. Indeed, this follows from the fact that the $\Gal(\bar{\eta}/\eta)$-action preserves $\NS(\sX_{\bar{\eta}})$ and factors through the natural map $\Gal(\bar{\eta}/\eta) \to \pi_1^\et(S, \bar{\eta})$, which is a surjection by the normality of $S$. The conclusion easily follows. Note that the hypothesis in (c) implies that $c_1 : \NS(\sX_{\bar{\eta}}) \to \H^2_\et(\sX_{\bar{\eta}}, \hat{\IZ}^p(1))$ is injective.

   (d) As the action of $\Gal(\bar{\eta}/\eta)$ on $\NS(\sX_{\bar{\eta}})_\IQ$ factors through $\pi_1^\et(S, \bar{\eta})$ and a finite quotient, the connectedness of $\mathrm{Mon}(R^2 f_* \IQ_\ell, s)$ implies that this action is in fact trivial, so that the natural map $\NS(\sX_\eta)_\IQ \to \NS(\sX_{\bar{\eta}})_\IQ$ is an isomorphism. Suppose that $\xi = \mathrm{sp}_\gamma(\beta)$ for some path $\gamma$ and $\beta \in \Pic(\sX_{\bar{\eta}})$. Then the class of $\beta$ in $\NS(\sX_{\bar{\eta}})_\IQ$ descends to a class $\bxi$ in $\NS(\sX_{\eta})_\IQ$. Choose a representative of $\bxi$ in $\Pic(\sX_\eta)$ and apply (a). 
\end{proof}

\begin{remark}
    In general, for (d) we cannot take $\bxi$ such that $\bxi_s = \xi$ as elements of $\Pic(\sX_s)$. Even when $S = \Spec(k)$ is a point, this cannot be guaranteed if $k$ is not algebraically closed: the morphism $\Pic(\sX) \into \Pic(\sX_{\bar{k}})^{\Gal_k}$ is in general not surjective ---the cokernel is a subgroup of the Brauer group $\mathrm{Br}(k)$ and does not always vanish. 
\end{remark}

\subsection{Local Deformation Spaces}

\subsubsection{}\label{sec: Def(xi) in char 0} Let $\kappa$ be a field of characteristic $0$, $R$ be a complete local $\kappa$-algebra with residue field $\kappa$. Let $S := \Spf(R)$ and $s$ be the special point. Suppose that $f : \sX \to S$ is a smooth projective family with geometrically connected fibers. Let $\xi \in \Pic(\sX_s)$ be any element. Then there exists a closed formal subscheme $\Def(\xi) \subseteq S$ characterized by the property that for every Artin $\kappa$-algebra $A$, an element $\wt{s}$ in $S(A)$ lies in $\Def(\xi)(A)$ if and only if $\xi$ deforms, possibly non-uniquely, to $\sX_{\wt{s}}$. It is not hard to construct $\Def(\xi)$ explicitly using the standard deformation theory of line bundles: Let $(\sH, \nabla, \Fil^\bullet \sH)$ be the filtered flat vector bundle $R^2 f_* \Omega^\bullet_{\sX/S}$. Then $c_{1, \dR}(\xi)$ extends uniquely to a horizontal section $\sigma$ of $\sH$, as $\sH$ admits a flat trivialization $(\sH, \nabla) \simeq \sH_s \tensor (R, d)$ (see for example the proof of \cite[Prop.~2.16]{BOBook}). Then $\Def(\xi) = \Spf(R/I)$ where $I$ is the minimal ideal such that the image of $\sigma$ vanishes in $(\sH / \Fil^1 \sH) \tensor_R (R/I)$. In particular, it is clear that $I$ is generated by $h^{0, 2} := \mathrm{rank\,} \sH / \Fil^1 \sH$ elements. \\

Our goal in this section is to prove the analogue of \ref{sec: Def(xi) in char 0} in positive and mixed characteristic. The difficulty to overcome in this case is that in general $R^2 f_* \Omega^\bullet_{\sX/S}$ does not admit a flat trivialization. 

\begin{definition}
\label{def: admissible}
    Let $S$ be a noetherian (possibly formal) scheme and $f: \sX \to S$ be a smooth projective morphism with geometrically connected fibers. We say that $\sX/S$ is an \textbf{admissible} family if the following conditions hold: 
    \begin{enumerate}[label=\upshape{(\roman*)}]
        \item the Hodge-de Rham spectral sequence degenerates at $E_1$ page; 
        \item the Hodge cohomology $\H^j(\Omega^i_{\sX/S}) := R^j f_* \Omega_{\sX/S}^i$ is locally free for every $(i, j)$. 
        \item for every prime $p$, the crystal of coherent sheaves $R^r f_{\IF_p, \cris *} \sO_{\sX_{\IF_p}/\IZ_p}$ on $\Cris(S_{\IF_p}/\IZ_p)$ is locally free for every $r$. 
    \end{enumerate}
    When $S$ is connected, we shall denote $\mathrm{rank\,} \H^j(\Omega^i_{\sX/S})$ by $h^{i, j}$. 
\end{definition}
\begin{remark}
    The above assumptions are used to make sure that various cohomology modules commute with base change: (i) and (ii) easily imply that for each $r$, $\H^r_\dR(\sX/S) := R^r f_* \Omega^\bullet_{\sX/S}$ is locally free. Moreover, using \cite[III~Cor.~12.7]{Hartshorne}, one quickly checks by induction that for any morphism $T \to S$ between noetherian (formal) schemes, the natural base change morphisms $\H^j(\Omega^i_{\sX/S}) \tensor_{\sO_S} \sO_T \to \H^j(\Omega^i_{\sX_T/T})$ and hence $\H^r_\dR(\sX/S) \tensor_{\sO_S} \sO_T \to \H^r_\dR(\sX_T/T)$ are isomorphisms for all $i, j, r$. Similarly, suppose that $T$ is a noetherian $\IF_p$-scheme over $S_{\IF_p}$ and there is PD morphism $T' \to S'$ between PD thickenings $T'$ of $T$ and $S'$ of $S_{\IF_p}$. Then the natural morphism $$(R^r f_{\IF_p, \cris *} \sO_{\sX_{\IF_p}/\IZ_p})_{S'} \tensor_{\sO_{S'}} \sO_{T'} \to (R^r f_{T, \cris *} \sO_{\sX_{T}/\IZ_p})_{T'}$$
    is an isomorphism. One may deduce this from the commutativity with base change on the level of derived categories \cite[Thm~7.8]{BOBook} and the universal coefficient theorem. 
\end{remark}   

The following lemma follows from a famous theorem of Deligne-Illusie \cite{DI}, standard base change theorem for coherent cohomology, and crystalline de-Rham comparison isomorphisms. 
\begin{lemma}
\label{lem: admissible for large primes}
    Let $L$ be a number field, $N \in \IN_{> 0}$, $\sS$ be a connected smooth scheme over $\sO_L[1/N]$ of finite type, and $f : \sX \to \sS$ be a smooth projective morphism with geometrically connected fibers. Then up to increasing $N$, $\sX$ is admissible. 
\end{lemma}

For the content below, let $k$ be an algebraically closed field of characteristic $p > 0$ and let $W = W(k)$. 

\begin{definition}
    Let $S$ be the $p$-adic completion of a noetherian $W$-scheme. A filtered F-crystal on $S$ is a crystal $\bH$ of vector bundles on $\Cris(S_k/W)$ equipped with a decreasing filtration $\Fil^\bullet$ on $\bH_S$ by direct summands and a Frobenius morphism $F : F^*_{S_k/W} \bH \to \bH$, where $F_{S_k/W}$ is the morphism on the topos $(S/W)_\cris$ induced by the absolute Frobenius morphism $F_{S_k}$ on $S_k$. 
\end{definition}

Here $\bH_S$ denotes the $\sO_S$-module obtained by evaluating $\bH$ on $S$, which is viewed as a (pro-)object of $\Cris(S_k/W)$. 

\begin{remark}
    We remark that, if $\bH = R^r f_{\cris *} \sO_{\sX/W}$ for an admissible family $f : \sX \to S_{\IF_p}$, then the Hodge filtration on $\bH_S \simeq R^r f_* \Omega^\bullet_{\sX/S}$ and the Frobenius action satisfy a compatibility condition (\cite[Thm~8.26]{BOBook}). We will only consider situations in which this condition is satisfied, but we will not make use of it, so we do not incorporate it into the definition above.
\end{remark}

\begin{proposition}
\label{prop: Frobenius trick}
    Let $R$ be a northerian complete local $k$-algebra. Set $S = \mathrm{Spf}(R)$ and let $s$ be the closed point. Let $\bH$ be an F-crystal over $\mathrm{Cris}(S/W)$. If $R$ is lci, then for any $m \in \IN_{>0}$ the map 
    \begin{equation}
    \label{eqn: restriction of Frobenius invariant}
        \Gamma((S/W)_\cris, \bH)^{F = p^m} := \{\alpha \in \Gamma((S/W)_\cris, \bH) \mid F(\alpha \tensor 1) = p^m \alpha \}  \to \Gamma((s/W)_\cris, \bH|_{s})
    \end{equation}
    is injective.
\end{proposition}

The proof below is essentially a version of ``Dwork trick'', which was used in \cite{BO}. 
\begin{proof}
    Let $\fm$ be the maximal ideal and $d$ be the Krull dimension of $R$ and choose a system of parameters $\{x_1, x_2, \cdots, x_d\}$. Let $\alpha = (\alpha_1, \cdots, \alpha_d) \in \IN_{>0}^{d}$ be a multi-index equipped with a partial order such that $\alpha \le \alpha'$ if and only if $\alpha_i \le \alpha'_i$ for every $i$. Let $I_\alpha := (x_1^{\alpha_1}, \cdots, x_d^{\alpha_d})$ so that $R = \varprojlim R/ I_\alpha$. Note that each $R/ I_\alpha$ is Artinian and lci. It suffices to show that the statement is true for each $R / I_\alpha$, so that we reduce to the case when $R$ itself is Artinian, i.e., $\fm$ is nilpotent.
    
    Choose a smooth $W$-scheme $\wt{S}$ such that $S$ admits a closed embedding into $\wt{S}$ and let $D$ be the PD envelope of $S$ in $\wt{S}$. Since $S$ is lci, $D$ (and hence $\bH_D$) has no $p$-torsion (cf. \cite[Rmk~1.10(i)]{BO}). This implies that $\Gamma((S/W)_\cris, \bH)$ has no $p$-torsion. On the other hand, we note that for $n \gg 0$, $F_S^n$ factors as $S \to s \stackrel{j}{\into} S$ for some morphism $S \to s$, where $j$ is the embedding and $F_S$ is the absolute Frobenius morphism. Suppose now that $\gamma \in \Gamma((S/W)_\cris, \bH)^{F = p^m}$ is an element which lies in the kernel of (\ref{eqn: restriction of Frobenius invariant}), i.e., $j^*(\gamma) = 0$. Since $F_S^n$ factors through $j$, this implies that $F^n(\gamma \tensor 1) = p^{mn} \gamma = 0$, so that $\gamma = 0$. 
\end{proof}

\begin{example}
    The assumption on $R$ being lci is necessary: Let $R = k[\![x, y]\!]/ (x^2, xy, y^2)$. Set $S = \Spf(R)$ and let $s$ be the closed point. Let $f : \sE \to S$ be a constant elliptic curve whose fiber over the closed point is ordinary. Set $\mathcal{A} = \sE \times_S \sE$. Since $\End (R^1 f_{\cris*} \sO_{\sE/W})(-1)$ is a direct summand of $\bH := R^2 (f \times_S f)_{\cris *} \sO_{\sA/W}$, the map $\Gamma((S/W)_\cris, \bH)^{F = p} \to \Gamma((s/W)_\cris, \bH|_{s})$ is not injective by \cite[(A.4)]{BO}. The problem is that the PD envelope of $S$ in a smooth $W$-scheme indeed has $p$-torsion. 
\end{example}

\begin{theorem}
\label{thm: Deligne generalized}
Let $R$ be a complete local $W$-algebra with residue field $k$ such that $R_k = R \tensor k$ is reduced and lci. Let $S = \Spec(R)$, $\hat{S} := \Spf(R)$ and $s$ be the special point. Suppose that $f : \sX \to S$ is an admissibe family. Let $\xi \in \Pic(\sX_s) = \shP_{\sX/S}(k)$ be any element. 

Let $\hat{\shP}^\xi$ be the formal completion of $\shP_{\sX/S}$ at the closed point $(s, \xi)$ and let $\Def(\xi) \subseteq \hat{S}$ be its scheme-theoretic image. Then $\hat{\shP}^\xi$ is formally smooth over $\Def(\xi)$ and the ideal $J \subseteq R$ defining $\Def(\xi)$ is generated by at most $h := \dim_k \H^2(\sO_{\sX_s})$ elements; moreover $\Def(\xi)$ only depends on the class of $\xi$ in $\NS(\sX_s)$. 
\end{theorem}

Note that the formal smoothness of $\hat{\shP}^\xi$ over $\Def(\xi)$ implies that as a closed formal subscheme of $\hat{S}$, $\Def(\xi)$ is characterized by the property that for any $A \in \mathsf{Art}_W$ (the category of Artin $W$-algebras) and $\wt{s} \in \hat{S}(A)$, $\wt{s}$ factors through $\Def(\xi)$ if and only if $\xi$ (possibly not uniquely) deforms to $\sX_{\wt{s}}$. This justifies the notation (cf. \ref{sec: Def(xi) in char 0}).

\begin{proof}
    For every $A \in \mathsf{Art}_W$, note that a point in $\shP_{\sX/S}(A)$ is of the form $(z, \zeta)$, where $z \in \hat{S}(A)$ and $\zeta \in \Pic(\sX_z)$ because $\sX \to S$ admits a section. Let $\Phi^\xi : \mathsf{Art}_W \to \mathsf{Set}$ be the functor defined by sending $A \in \mathsf{Art}_W$ to the set of pairs $(z, \alpha)$, where $z \in \hat{S}(A)$ and $\alpha = c_{1, \cris}(\zeta) \in \Gamma((A/W)_\cris, \bH|_z)$ for some $\zeta \in \Pic(\sX_z)$ such that $\zeta_s \sim \xi$. Equivalently, the pair $(z, \alpha)$ is an equivalence class of pairs $(z, \zeta)$ such that $\zeta \in \Pic(\sX_z)$ is an element with $\zeta_s \sim \xi$, and $(z, \zeta)$ is equivalent to $(z, \zeta')$ if and only if $c_{1, \cris}(\zeta) = c_{1, \cris}(\zeta')$. 
    
    We start by proving two lemmas: 

    \begin{lemma}
\label{lem: injectivity}
    For every $A \in \mathsf{Art}_W$, the projection $\Phi^\xi(A) \to \hat{S}(A)$ is injective. 
\end{lemma}
\begin{proof}
    It suffices to show that for each $z \in \hat{S}(A)$, given $\zeta, \zeta' \in \Pic(\sX_z)$ such that $\zeta_s \sim \zeta_s'$, we have $c_{1, \cris}(\zeta) = c_{1, \cris}(\zeta')$ in $\Gamma((A/W)_\cris, \bH|_{z})$. Let $\delta := \zeta' \tensor \zeta^\vee$. Then $\delta_s \in \shP^0_{\sX_s/k}(k)$ and it suffices to show that $c_{1, \cris}(\delta) = 0$. 
    
    Recall that for every alg-geometric point $t \to S$, $\H^2_\cris(\sX_t/ W(k(t)))$ is torsion-free. This implies that the fiber $\shP^0_{\sX/S}|_t = \shP^0_{\sX_t/k(t)}$ is reduced (or equivalently smooth) by \cite[II~Prop.~5.16]{Illusie}; moreover, the dimension of $\shP^0_{\sX/S}|_t$ is $\dim \H^1(\sO_{\sX_t})$, which by assumption is independent of $t$. 
    
    Since $S$ is reduced, by \ref{prop: Kleiman} $\shP^0_{\sX/S}$ is smooth, so that there exists a deformation $\wt{\delta} \in \Pic(\sX) = \shP_{\sX/S}(S)$ such that $\wt{\delta}_z = \delta$. Then by the functoriality of the Chern class map (see e.g., \cite[Thm~2.4(i)]{Berthelot-Illusie}) $c_{1, \cris}(\delta)$ is the image of $c_{1, \cris}(\wt{\delta})$ under the natural restriction map $\Gamma((S/W)_\cris, \bH) \to \Gamma((A/W)_\cris, \bH|_z)$. Therefore, it suffices to show that $c_{1, \cris}(\wt{\delta}) = 0$. Now we note that $c_{1, \cris}(\wt{\delta})$ lies in the kernel of the map $\Gamma((S/W)_\cris, \bH)^{F = p} \to \Gamma((s/W)_\cris, \bH|_s)$, which is $0$ by \ref{prop: Frobenius trick}. 
\end{proof}

\begin{lemma}
\label{lem: formal smoothness}
    The natural projection $\hat{\shP}^\xi \to \Phi^\xi$ is formally smooth. That is, given any surjection $A \to A'$ whose kernel is square zero, the natural morphism 
    $$ \hat{\shP}^\xi(A) \to \hat{\shP}^\xi(A') \times_{\Phi^\xi(A')} \Phi^\xi(A) $$
    is surjective. 
\end{lemma}

\begin{proof}
    Concretely, the above amounts to saying that given $z \in \hat{S}(A)$ and line bundles $\zeta_1, \zeta_2 \in \Pic(\sX_{z'})$, where $z' \in \hat{S}(A')$ is the image of $z$, then if $\zeta_1$ extends to $\sX_z$, so does $\zeta_2$. By \cite[Prop.~1.12]{OgusCrystals}, $\zeta_i$ extends to $\sX_z$ if and only if $c_{1, \cris}(\zeta_i)_z \in (\bH|_{z'})_A \iso \H^2_\dR(\sX_{z}/A)$ lies in the $\Fil^1$ part of Hodge filtration. However, \ref{lem: injectivity} implies that $c_{1, \cris}(\zeta_1) = c_{1, \cris}(\zeta_2)$, so the conclusion follows. 
\end{proof}

    
Back to the proof of \ref{thm: Deligne generalized}, given morphisms $A' \to A$ and $A'' \to A$ in $\mathsf{Art}_W$, we have a commutative diagram 
    \[\begin{tikzcd}
	{\hat{\shP}^\xi(A' \times_A A'')} & {\hat{\shP}^\xi(A') \times_{\hat{\shP}(A)} \hat{\shP}(A'')} \\
	{\Phi^\xi(A' \times_A A'')} & {\Phi^\xi(A') \times_{\Phi^\xi(A)} \Phi^\xi(A'')} \\
	{\hat{S}(A' \times_A A'')} & {\hat{S}(A') \times_{\hat{S}(A)} \hat{S}(A'')}
	\arrow["\sim", from=1-1, to=1-2]
	\arrow["\pi_1", two heads, from=1-1, to=2-1]
	\arrow[from=2-1, to=2-2]
	\arrow["\pi_2", two heads, from=1-2, to=2-2]
	\arrow["\iota_1", hook, from=2-1, to=3-1]
	\arrow["\sim", from=3-1, to=3-2]
	\arrow["\iota_2", hook, from=2-2, to=3-2]
    \end{tikzcd}\]
     The top and bottom horizontal arrows are isomorphisms because $\hat{\shP}^\xi$ and $\hat{S}$ are already pro-representable. The surjectivity of $\pi_1$ is clear by definition, and the injectivity of $\iota_i$'s follow from \ref{lem: injectivity}. The surjectivity of $\pi_2$ follows easily from the argument of \ref{lem: formal smoothness} by iteratively considering square zero extensions. 
     

     
     Now it is easy to infer that the middle horizontal arrow is an isomorphism. Note that the tangent space $\Phi^\xi(k[\epsilon]/(\epsilon^2))$ is finite dimensional as it admits a surjection from $\hat{\shP}^\xi(k[\epsilon]/(\epsilon^2))$. Schlessinger's criterion \cite[Thm~2.11]{Schlessinger} implies that $\Phi^\xi$ is pro-representable by a formal scheme $\hat{Z} := \Spf(R')$. By \cite[Lem.~2.1.7]{HO_conjecture}, the natural morphism $\hat{Z} \to \hat{S}$ is a closed immersion, and by \ref{lem: formal smoothness} the morphism $\hat{\shP}^\xi \to \hat{Z}$ is formally smooth. This implies that $\hat{Z}$ is indeed the scheme-theoretic image of $\hat{\shP}^\xi$, i.e., $\hat{Z} = \Def(\xi)$. Moreover, it is clear from the definition of the functor $\Phi^\xi$ that $\Phi^\xi = \Phi^{\xi'}$ if $\xi \sim \xi'$, so that $\Def(\xi)$ only depends on the class of $\xi$ in $\NS(\sX_s)$.


     It remains to show that the ideal $J = \ker(R \to R')$ defining $\hat{Z} \subseteq \hat{S}$ is generated by $h$ elements, which follows from a simple adaptation of Delgine's arguments for \cite[Prop.~1.5]{Del02}. By choosing a section of $\hat{\shP}^\xi \to \hat{Z}$ we obtain a line bundle $\wt{\xi} \in \Pic(\sX_{\hat{Z}})$ which deforms $\xi$. Let $\frakm \subseteq R$ be the maximal ideal. The obstruction to deforming $\wt{\xi}$ to $\hat{Z}'' := \Spf(R / \frakm J)$ is given by an element $\frak{o} \in \H^2(X_s, J / \frakm J) \iso \H^2(\sO_{\sX_s}) \tensor J / \frakm J$. By choosing a basis for $\H^2(\sO_{\sX_s})$, $\frak{o}$ gives us $h$ elements in $J / \frakm J$. Choose liftings $f_1, \cdots, f_h$ of these elements to $J$ and set $\Sigma \subseteq S$ to be the ideal $\frakm J + (f_1, \cdots, f_h)$. Then by construction $\hat{Z} \subseteq \Sigma \subseteq \hat{Z}''$. Since $\wt{\xi}$ deforms to $\Sigma$, there is a section $\Sigma \to \hat{\shP}^\xi$. As $\hat{Z}$ is the scheme-theoretic image, $\Sigma \subseteq \hat{Z}$. Therefore, we must have $\hat{Z} = \Sigma$ and hence $\frakm J + (f_1, \cdots, f_h) = J$. By Nakayama's lemma $J = (f_1, \cdots, f_h)$. 
\end{proof}

Below are some simple observations for future reference. 
\begin{lemma}
\label{lem: Def(xi)}
    In \ref{thm: Deligne generalized}, for any $m \in \IN$ with $p \nmid m$, $\Def(\xi) = \Def(m \xi)$. 
\end{lemma}
\begin{proof}
    We run a simple induction. Suppose that $R' \to R$ is a surjection in $\mathsf{Art}_W$ with square zero kernel and $\Def(\xi)(R) = \Def(m \xi)(R) \subseteq \hat{S}(R)$. We just need to show that $\Def(\xi)(R') = \Def(m \xi)(R')$. That $\Def(\xi)(R') \subseteq \Def(m \xi)(R')$ is automatic, so we only need to show the reverse inclusion. 
    
    Take $z' \in \Def(m \xi)(R')$ and a deformation $\xi'_{z'}$ of $m\xi$ to $z'$. Let $z := z' \tensor_{R'} R$. Then by assumption $z \in \Def(\xi)(R)$ so there exists a deformation $\xi_z$ of $\xi$ over $z$. Setting $\xi'_z := \xi'_{z'} \tensor_{R'} z$. By \ref{lem: injectivity}, $c_{1, \cris}(m \xi_z) = c_{1, \cris}(\xi'_z)$. Therefore, by \cite[Prop.~1.12]{OgusCrystals}, the fact that $\xi'_z$ deforms over $z'$ implies that $m \xi_z$ deforms to $R'$ as well. Applying this proposition again, as $c_{1, \cris}(m \xi_z) = m c_{1, \cris}(\xi_z)$ and $m$ is a unit, $\xi_z$ also deforms to $z'$, i.e., $z' \in \Def(\xi)(R')$. 
\end{proof}
\begin{lemma}
\label{lem: Galois descent}
    Let $\kappa$ be a perfect field and $X$ be an admissible smooth proper variety over $\kappa$. Then $\NS(X)_{\mathrm{tf}} \to (\NS(X_{\bar{\kappa}})_{\mathrm{tf}})^{\Gal_\kappa}$ is an isomorphism. If $\mathrm{char\,} \kappa = p > 0$, then $\NS(X)$ has no $p$-torsion. 
\end{lemma}
 The subscript ``tf'' stands for the torsion-free quotient.
\begin{proof}
It is well known that the statement is true for $\NS(X)_\IQ$, i.e., the natural map $\NS(X)_{\IQ} \to (\NS(X_{\bar{\kappa}})_\IQ)^{\Gal_\kappa}$ is an isomorphism (see e.g., \cite[Prop.~6.1(b)]{PTL}). For every $\ell \neq \mathrm{char\,}k$, $\NS(X) \tensor \IZ_\ell \to \H^2_\et(X_{\bar{\kappa}}, \IZ_\ell(1))$ is injective and has torsion-free cokernel. Similarly, if $\mathrm{char\,} k = p > 0$, as $\H^2_\cris(X/W(\kappa))$ is $p$-torsion free, $\NS(X)$ has no $p$-torsion and the map $\NS(X) \tensor \IZ_p \to \H^2_\cris(X/W(\kappa))^{F = p}$ is injective and has torsion-free cokernel. Hence the statement for $\NS(X)_\IQ$ implies the statement for $\NS(X)_{\mathrm{tf}}$. 
\end{proof}

\begin{lemma}
\label{globalimagelem2}
    Suppose that $S$ is a connected smooth scheme over $W$ with generic point $\eta$, $s \in S(k)$ is a closed point, and $\shP(S, s, \xi)$ dominates $S$. Then $\xi$ deforms to the formal completion of $S$ at $s$. 
\end{lemma}
\begin{proof}
    Since $\shP(S, s, \xi)$ dominates $S$, $\xi$ is obtained by specializing a line bundle on $\sX_{\bar{\eta}}$. Up to replacing $S$ by a connected \'etale cover and $s$ by a lift, we may assume that $\NS(\sX_{\bar{\eta}})_{\mathrm{tf}} = \NS(\sX_\eta)_{\mathrm{tf}}$. As $\NS(\sX_s)$ has no $p$-torsion, for some $m \in \IN$ with $p \nmid m$, $m \xi$ lies in the image of $\NS(\sX_\eta) \to \NS(\sX_s)$. As every line bundle on $\sX_\eta$ extends to $\sX$, $m\xi$ (and hence $\xi$, by \ref{lem: Def(xi)}) deforms to the formal completion of $S$ at $s$. 
\end{proof}

\section{Jets and Period Maps}

\subsection{Jet Spaces}
\label{jetspsec}

Our arguments will use \emph{jet spaces}, which we now define. Given a ring $R$, we set 
\[ A^{d}_{r, R} = R[t_{1}, \hdots, t_{d}]/(t_{1}, \hdots, t_{d})^{r+1} , \]      
and write $\mathbb{D}^{d}_{r, R}$ for $\Spec (A^{d}_{r, R})$. If $S$ is an $R$-scheme, the \emph{jet space} $J^{d}_{r} S$ associated to the non-negative integers $d, r \geq 0$ and $S$ is defined to be the $R$-scheme representing the functor $\textsf{Sch}_{R} \to \textsf{Set}$ given by
\[ T \mapsto \Hom_{R}(T \times_{R} \mathbb{D}^{d}_{r, R}, S), \hspace{1.5em} [T \to T'] \mapsto [\Hom_{R}(T' \times_{R} \mathbb{D}^{d}_{r, R}, S) \to \Hom_{R}(T \times_{R} \mathbb{D}^{d}_{r, R}, S)] , \]
where the natural map $\Hom_{R}(T' \times_{R} \mathbb{D}^{d}_{r, R}, S) \to \Hom_{R}(T \times_{R} \mathbb{D}^{d}_{r, R}, S)$ obtained by pulling back along $T \times_{R} \mathbb{D}^{d}_{r, R} \to T' \times_{R} \mathbb{D}^{d}_{r, R}$. To check representability of the functor one may produce as in \cite[Prop 2.2]{urbanik2021transcendence} an isomorphism to the functor defining the Weil restriction functor $\textrm{Res}_{\mathbb{D}^{d}_{r,R}/R}(\mathbb{D}^{d}_{r,R} \times_{R} S)$, and hence reduce to the corresponding representability questions for Weil restrictions. For our purposes it will be enough that $J^{d}_{r} 
 S$ exists when $S$ is a quasi-projective $R$-scheme.

 A few basic properties of the functors $J^{d}_{r}$ which are useful and immediately checked from the definition are as follows:
 \begin{itemize}
     \item[(i)] given a ring map $R \to R'$, we have that $J^{d}_{r} (S_{R'}) = (J^{d}_{r} S)_{R'}$ (c.f. \cite[Prop 6.2.2(2)]{zbMATH07565706} for the corresponding statement for Weil restrictions); 
     \item[(ii)] each map $\mathbb{D}^{d}_{r, R} \to \mathbb{D}^{d'}_{r',R}$ between infinitesimal disks induces a morphism of functors $J^{d'}_{r'} \to J^{d}_{r}$ by pre-composition; and
     \item[(iii)] given a further composition $\mathbb{D}^{d'}_{r',R} \to \mathbb{D}^{d''}_{r'',R}$ the maps $J^{d'}_{r'} \to J^{d}_{r}$, $J^{d''}_{r''} \to J^{d'}_{r'}$ and $J^{d''}_{r''} \to J^{d}_{r}$ are related by natural transformations in the obvious way.
 \end{itemize}

Denote by $\pi^r_{r - k}$ the natural map $J^{d}_{r} \to J^{d}_{r-k}$ induced by the natural reduction map $A^d_{r, R} \to A^d_{r -k , R}$ sending each $t_i$ to $t_i$. If $r = k$, we simply write $\pi^r_0 : J^d_r S \to J^d_0 S = S$ as $\pi$. Given $j \in J^d_r S(R)$, we also write $j(0)$ for the image $\pi(j)$ in $S(R)$.

Lastly, we write $J^{d}_{r,nd} S \subset J^{d}_{r} S$ for the open subscheme of \emph{non-degenerate} jets, and $J^{d}_{r,c} S \subset J^{d}_{r} S$ for the closed subscheme of constant jets; see \cite[Def. 3.4]{urbanik2022algebraic} and \cite[\S3]{urbanik2022algebraic} for precise definitions. The field-valued points of $J^{d}_{r,nd} S$ can be identified with maps from $\ID^{d}_{r}$ to $S$ which induce embeddings on the level of tangent spaces, and the field-valued points of $J^{d}_{r,c} S$ can be identified with maps from $\ID^{r}_{d}$ to $S$ which are constant. The construction of both $J^{d}_{r,nd} S$ and $J^{d}_{r,c} S$ is compatible with base-change in the sense of (i) above.

\subsection{Period Jets}
\label{sec: period jets}

In this section we fix an integral domain $R$, a smooth $R$-scheme $S$ and the data $(\sH, \Fil^{\bullet}, \nabla)$ where $\sH$ is a vector bundle of rank $m \in \IN_{> 0}$, $\Fil^\bullet$ is a decreasing filtration on $\sH$ and an $R$-linear integrable connection $\nabla : \sH \to \sH \tensor \Omega^1_{S/R}$. We additionally write $\che{L}$ for the $R$-scheme parameterizing flags on $R^{m}$ whose graded quotients are locally free and are of the same rank as those of $\Fil^{\bullet}$ (c.f. \cite[\S1.2]{urbanik2022algebraic}). Below we shall fix some $r \in \IN$ and assume that $r!$ is invertible in $R$.

\begin{definition}
\label{rlimpdef}
    Let $k$ be a field and consider $s \in S(k)$. Let $\Delta^r_s := \Spec(\sO_{S_k, s}/ \fm^{r + 1})$, where $\fm$ is the maximal ideal of the local ring $\sO_{S_k, s}$. Given a frame (i.e., a basis) $b^\bullet = \{ b^1, \cdots, b^m \}$ of $\sH_s$, there is a unique way of extending $b^\bullet$ to a set of horizontal generators of $\sH|_{\Delta^r_s}$, which induces a horizontal trivialization $(\sH, \nabla)|_{\Delta^r_s} \iso (\sO_{\Delta^r_s}, d)^{\oplus m}$.\footnote{This fact can be dug out from \cite[Lem.~2.6]{urbanik2022algebraic}. A more direct alternative argument is given in the appendix.} Then the filtration $\Fil^\bullet \sH$ restricts to a filtration on $(\sO_{\Delta^r_s})^{\oplus m}$, which gives rise to a morphism $\psi_{b^\bullet} : \Delta^r_s \to \che{L}_k$. We call $\psi_{b^\bullet}$ the \textit{$r$th order (local) infinitesimal period morphism ($r$-limp)} associated to the frame $b^\bullet$. 
\end{definition}

\subsubsection{} \label{sec: define P^d_r} We will study $r$-limps through the maps of jets they induce. For any $d \in \IN$, consider the $\GL_{m,R}$-torsor $\gamma^d_r : \sP^d_r \to J^d_r S$ obtained as the base-change along $J^d_r S \to S$ of the frame bundle of $\sH$. Then for every field $k$, there is a natural bijection between $\sP^d_r(k)$ and pairs of the form $(j, b^\bullet)$, where $j$ is a morphism $\ID^d_{r, k} \to S_k$ and $b^\bullet$ is a frame on $\sH_{s}$ for $s := j(0)$. Note that $j : \ID^d_{r, k} \to S_k$ uniquely factors through $\Delta^r_s$. We can then define a natural map $\alpha^d_r(k) : \sP^d_r(k) \to J^d_r \che{L} (k)$ by sending each $(j, b^\bullet)$ to the composition  
$$\ID^d_{r, k} \stackrel{j}{\to} \Delta^r_s \stackrel{\psi_{b^\bullet}}{\to} \che{L}_k. $$

The following result says that the maps $\alpha^d_r(k)$ ``glue'' in some natural way, and can be obtained by unpacking \cite[Thm~1.14]{urbanik2022algebraic} and its proof: 

\begin{theorem}
\label{bigjetthm}
Associated to the triple $(\sH, \Fil^\bullet, \nabla)$, there exists a $\GL_{m, R}$-equivariant morphism $\alpha^d_r : \sP^d_r \to J^d_r \che{L}$ whose induced map on $k$-points for each algebraically closed field $k$ is $\alpha^d_r(k)$ as above. The resulting morphism 
\[ \eta^{d}_{r} : J^{d}_{r} S \to \GL_{m, R} \backslash J^{d}_{r} \che{L} \]
of Artin stacks over $R$ is compatible with 
\begin{itemize}
\item[-] base-change along maps $R \to R'$ of $\IZ[1/r!]$-algebras;
\item[-] pullback along maps $g : S' \to S$ of smooth $R$-schemes, in the sense that $\eta^{d}_{r} \circ J^{d}_{r} g$ is the map associated to $(g^{*} \sH, g^{*} \Fil^{\bullet}, g^{*} \nabla)$; and
\item[-] maps of functors $J^{d}_{r} \to J^{d'}_{r'}$ for $r' \leq r$ induced by maps $\mathbb{D}^{d'}_{r',R} \to \mathbb{D}^{d}_{r,R}$, in the sense that the natural square involving $J^{d}_{r} S, J^{d'}_{r'} S, \GL_{m} \backslash J^{d}_{r} \che{L}$ and $\GL_{m} \backslash J^{d'}_{r'} \che{L}$ commutes. 
\end{itemize}
\end{theorem}

\section{Proofs of Theorems}
\subsection{General Setup} \label{sec: general setup}
Let $R \subseteq \IC$ be a sub-$\IZ[1/N]$-algebra for some $N \in \IN_{>0}$ and let $F \subseteq \IC$ be its fraction field. Suppose that $\sS$ is a smooth quasi-projective $R$-scheme such that $\sS_\IC$ is connected and $f : \sX \to \sS$ be an admissible family of relative dimension $d_0$. Choose a polarization (i.e., a relatively ample line bundle) $\bxi_0$ on $\sX/\sS$. Let $\sH$ be the filtered flat vector bundle $R^2 f_* \Omega^\bullet_{\sX/\sS}$ on $\sS$, equipped with its Hodge filtration $\Fil^\bullet$ and Gauss-Manin connection $\nabla$. We let $\eta$ be the generic point of $\sS$ and $\bar{\eta}$ is a geometric point over $\eta$.

Let $\IV = (\IV_B, \IV_\dR)$ be the VHS given by on $\IV_B := R^2 f_{\IC *} \IQ$ and $\IV_\dR := \sH|_{\sS_\IC}$. Note that $\IV_B$ has a natural integral structure given by $(R^2 f_{\IC*} \IZ)_{\mathrm{tf}}$, where ``tf'' standards for the torsion-free quotient. Using $\bxi_0$, one defines a symmetric bilinear pairing on $\IV$ via the usual formula $\< x, y \> = x \cup y \cup c_1(\bxi_{0, \IC})^{d_0 - 2}$, where $x, y$ are local sections. As $\bxi_0$ is defined over $S$, using the same formula one defines a symmetric bilinear pairing on $\sH$. We impose the following assumptions and notations. 
\begin{enumerate}[label=\upshape{(\roman*)}]
    \item $f_\IC$ has large monodromy and sufficiently large period image (recall \ref{def: large monodromy}). 
    \item $\Mon(\IV_B)$ is connected. 
    \item The pairing on $\sH$ induced by $\bxi_0$ is self-dual. 
    \item Every prime dividing the order of $\NS(\sX_{\bar{\eta}})_{\mathrm{tor}}$ is a prime factor of $N$, and $\NS(\sX_{\bar{\eta}})[N^{-1}]$ equipped with the pairing induced by $\bxi_0$ is self-dual as a quadratic $\IZ[1/N]$-lattice. 
    \item The natural maps $\NS(\sX_\eta)_\tf \to \NS(\sX_{\bar{\eta}})_\tf$ and $\NS(\sX_\eta)[N^{-1}] \to \NS(\sX_{\bar{\eta}})[N^{-1}]$ are isomorphisms. 
    \item The relative de Rham Chern class map sends\footnote{Choose a section of the composition $\Pic(\sX) \to \Pic(\sX_\eta) \to \NS(\sX_\eta)$. The choice is unimportant.} $\NS(\sX_\eta) \tensor_\IZ \sO_\sS$ isomorphically onto an orthogonal direct summand $\sH^0$ of $\sH$.
\end{enumerate}


Note that $\sH^0$ is a filtered flat sub-bundle of $\sH$. In fact $\sH^0 \subseteq \Fil^1 \sH$. Let $\eta_\IC$ be the generic point of $\sS_\IC$. Then the relative Chern class of $\NS(\sX_{\eta_\IC})_\tf$ gives rise to a sub-VHS $\IV^0 = (\IV^0_B, \IV^0_\dR)$ of $\IV$ such that $\IV^0_\dR = \sH^0|_{\sS_\IC}$. We denote by $\IV^2 = (\IV^2_B, \IV^2_\dR)$ and $\sH^2$ the orthogonal complements of $\IV^0$ and $\sH^0$ in $\IV$ and $\sH$ respectively, and write the natural projections $\IV \to \IV^2$ and $\sH \to \sH^2$ both as $\pi_2$. We write $h^{0, 2}$ for $h^{0, 2}(\IV) = h^{0, 2}(\IV^2)$. Note that as $\IV^0$ is a trivial VHS, $\Mon(\IV_B, s) = \Mon(\IV^2_B, s)$ for every $s \in \sS(\IC)$. 

\begin{definition}
\label{def: power-primitive} 
    Let $s \to \sS$ be a geometric point and let $p := \mathrm{char\,}k(s)$. Here we allow $p = 0$, in which case $\IZ_{(p)} = \IQ$. Recall that there is a well defined specialization map $\NS(\sX_\eta) \to \NS(\sX_s)$ and denote the image by $\NS(\sX_\eta)_s$. Then since by assumption (iv) 
    $\NS(\sX_{\eta}) \tensor \IZ_{(p)}$ is a torsion-free self-dual quadratic lattice over $\IZ_{(p)}$, $\NS(\sX_{\eta}) \tensor \IZ_{(p)}$ maps isomorphically onto $\NS(\sX_{\eta})_s \tensor \IZ_{(p)}$. Moreover, as $\NS(\sX_{\eta})_s \tensor \IZ_{(p)}$ is self-dual, it has to be an orthogonal direct summand, and there is an orthogonal decomposition 
    \begin{equation}
    \label{eqn: decomp over Zp}
        \NS(\sX_s) \tensor \IZ_{(p)} = [\NS(\sX_{\eta})_{s} \tensor \IZ_{(p)}] \oplus [\NS(\sX_{\eta})_{s} \tensor \IZ_{(p)}]^\perp.
    \end{equation} We make the following definition: For $\xi \in \Pic(\sX_s)$, we say that $\xi$ is \textit{power-primitive} if its class lies in $[\NS(\sX_{\eta})_{s} \tensor \IZ_{(p)}]^\perp$ and is nonzero mod $p$.
\end{definition}

\begin{lemma}
    \label{lem: primitive line bundle} Consider the setting in \ref{def: power-primitive} and write $k$ for $k(s)$. 
    \begin{enumerate}[label=\upshape{(\alph*)}]
        \item The map $c_{1, \dR} \tensor k : \NS(\sX_s) \tensor_\IZ k \to \sH_s = \H^2_\dR(\sX_s/k)$ sends $\NS(\sX_\eta)_s \tensor_\IZ k$ isomorphically onto $\sH^0_s \subseteq \sH_s$. 
        \item If $\xi \in \Pic(\sX_s)$ is power-primitive, then $c_{1, \dR}(\xi) \in \sH_s$ is nonzero and lies in $\sH^2_s \subseteq \sH_s$. 
        \item For every $\xi \in \NS(\sX_s)$, there exists some $m \in \IN \smallsetminus p \IN$ such that $m \xi = \xi_0 + p^h \xi'$ for some $\xi_0$ whose class lies in $\NS(\sX_\eta)_s$, $h \in \IN$ and a power-primitive $\xi'$.  
    \end{enumerate}
\end{lemma}
\begin{proof}
    (a) follows from the definition of $\sH^0$ and the commutativity of taking de Rham Chern class with pullback. (b) As we assumed that $\sX \to \sS$ is admissible, the Hodge-de Rham spectral sequence of $\sX_s$ degenerates at $E_1$-page and $\H^*_\cris(\sX_s/W(k))$ is torsion-free. Then by \cite[Rmk~3.5]{Del02} (see also \cite[Cor.~5.18]{Illusie} and \cite[Cor.~1.4]{OgusCrystals}), the  map $c_{1, \dR} : \NS(\sX_s) \tensor \IF_p \to \H^2_\dR(\sX_s/k)$ is injective. This implies that $c_{1, \dR}(\xi)$ is nonzero. Moreover, from (a) it is clear that $c_{1, \dR}(\xi)$ is orthogonal to $\sH^0_s$ and hence lies in $\sH^2_s$. (c) is a direct consequence of (\ref{eqn: decomp over Zp}). 
\end{proof}

We set $\che{L}$ to be the $R$-scheme parameterizing flags on $R^{m}$ whose graded pieces are locally free and are of the same rank as those of $\IV^2$. Let $s$ be a point on $\sS$ valued over a field $k = k(s)$ and let $\Delta^r_s := \Spec(\sO_{\sS_k, s}/\fm^{r + 1})$, where $\fm$ is the maximal ideal of the local ring $\sO_{\sS_k, s}$. Then for a chosen basis $b^\bullet$ of $\sH^2_s$, we have a $r$-limp map $\psi_{b^\bullet} : \Delta^r_s \to \check{L}_k$ for any $r > 1$ such that $r!$ is invertible in $k$ (see \ref{rlimpdef}). 

Below we shall relate this $r$-limp map to the Kodaira-Spencer map. 

\subsubsection{} The map $T_s \sS_{k} \to \Hom(\H^1(\Omega_{\sX_s}), \H^2(\sO_{\sX_s}))$ that we call the Kodaira-Spencer map is given by the composition
\begin{equation}
\label{ksmap1}
T_s \sS_{k} \to \H^1(T_{\sX_s}) \to \Hom(\H^1(\Omega_{\sX_s}), \H^2(\sO_{\sX_s}))
\end{equation} 
where the first map exists because $H^1(T_{X_s})$ classifies deformations of $\sX_s$ over the dual numbers $k[\varepsilon]/(\varepsilon^2)$ and the second map is given by the cup product pairing. In fact, using cup product again we can extend to a map 
\begin{equation}
\label{ksmap}
T_s \sS_{k} \to \H^1(T_{\sX_s}) \to \Hom(\H^1(\Omega_{\sX_s}), \H^2(\sO_{\sX_s})) \oplus \Hom(\H^0(\Omega_{\sX_s}^2), \H^1(\Omega^1_{\sX_s})).
\end{equation} 

We now relate the above notion of Kodaira-Spencer mapping to its Hodge-theoretic variant. Given a vector space $V$ over a field and a flag variety $\sF$ parameterizing two-step flags $F^{\bullet} = \{ V = F^0 V \supseteq F^1 V \supseteq F^2 V \}$ on $V$, its tangent space $T_{F^\bullet} \sF$ at $F^\bullet$ is naturally identified with 
\[ (\Hom(F^1 / F^2, F^0/F^1) \oplus \Hom(F^2, F^1/F^2)) \oplus \Hom(F^2, F^0/F^1) \]
(c.f. \cite[Chap II]{zbMATH06031035}). When $V$ is taken to be $\sH_s = \H^2_\dR(\sX_s/k)$, and $F^\bullet$ is taken to be the Hodge filtration $\Fil^\bullet$, the variation of Hodge filtrations on $\sH$ at $s$ gives us a map $T_s \sS_{k} \to T_{F^\bullet} \sF$, which is also oftentimes called the Kodaira-Spencer map. Griffiths transversality guarantees that this map factors through the first two summands, so we obtain a map 
\begin{equation}
\label{grifftransversedecomp}
T_s \sS_{k} \to \Hom(\Fil^1 \sH_s / \Fil^2 \sH_s, \sH_s / \Fil^1 \sH_s) \oplus \Hom(\Fil^2 \sH_s, \Fil^1 \sH_s / \Fil^2 \sH_s).
\end{equation}
In fact, this agrees with (\ref{ksmap}) via the natural identifications $\Fil^2 \sH_s = \H^0(\Omega^2_{\sX_s})$, $\Fil^1 \sH_s / \Fil^2 \sH_s = \H^1(\Omega^1_{\sX_s})$, and $\sH_s / \Fil^1 \sH_s = \H^2(\sO_{\sX_s})$ provided by the degeneration of Hodge-de Rham spectral sequence of $\sX_s$ at $E_1$ page. Deligne explained this in \cite[(2.3.12)]{Del02} in the context of deforming K3 surfaces but the argument easily generalizes. Moreover, since the pairing on $\sH$ induced by the polarization $\bxi_0$ is horizontal, and under this pairing $\H^1(\Omega_{\sX_s}^1)$ is self-dual and $\H^2(\sO_{\sX_s})$ is dual to $\H^0(\Omega_{\sX_s}^2)$, the two factors $T_s \sS_{k} \to \Hom(\H^1(\Omega_{\sX_s}), \H^2(\sO_{\sX_s}))$ and $T_s \sS_{k} \to \Hom(\H^0(\Omega_{\sX_s}^2), \H^1(\Omega^1_{\sX_s}))$ of (\ref{ksmap}) are dual to each other up to a sign (cf. \cite[(2.3.14)]{Del02}). Therefore, (\ref{ksmap}) is remembered by (\ref{ksmap1}). 

Suppose now we take $V$ to be $\sH^2_s$ and $F^\bullet$ to be the filtration on $\sH^2_s$, which is obtained by restricting that on $\sH_s$. Since the orthogonal complement $\sH^0$ of $\sH^2$ in $\sH$ lies completely in $\Fil^1 \sH$, $\Fil^2 \sH^2 = \Fil^2 \sH$ and $\sH^2 / \Fil^1 \sH^2 = \sH / \Fil^1 \sH$. Moreover, as $\sH^0 = \NS(\sX_\eta) \tensor_\IZ \sO_\sS$ is generated by horizontal global sections, one easily checks that the map (\ref{grifftransversedecomp}) factors through 
\begin{equation}
\label{grifftransversedecomp2}
T_s \sS_{k(s)} \to \Hom(\Fil^1 \sH^2_s / \Fil^2 \sH^2_s, \sH^2_s / \Fil^1 \sH^2_s) \oplus \Hom(\Fil^2 \sH^2_s, \Fil^1 \sH^2_s / \Fil^2 \sH^2_s) \subseteq T_{F^\bullet} \sF, 
\end{equation}
where the codomain of the above is naturally viewed as a submodule of the codomain of (\ref{grifftransversedecomp}). The chosen basis $b^\bullet$ of $\sH^2_s$ gives us an identification $\sF \stackrel{b^\bullet}{\simeq} \che{L}_k$, and there is a commutative diagram
\begin{equation}
    \label{eqn: diag for KS}
    \begin{tikzcd}
	{T_s \sS_{k}} & {T_{F^\bullet} \sF} \\
	{T_s \Delta^r_s} & {T_{\psi^\bullet(s)} \che{L}_k}
	\arrow["(\ref{grifftransversedecomp2})", from=1-1, to=1-2]
	\arrow[equal, from=1-1, to=2-1]
	\arrow["{\stackrel{b^\bullet}{\simeq}}", from=1-2, to=2-2]
	\arrow["T_s \psi_{b^\bullet}", from=2-1, to=2-2]
\end{tikzcd}.
\end{equation}

\subsection{Construction of the Atypical Locus}
Consider the set-up in \S\ref{sec: general setup}. We apply the constructions in \S\ref{sec: period jets} to the filtered flat vector bundle $\sH^2$ considered in \S\ref{sec: general setup}. Let $d = \dim_{R} \sS - (h^{0, 2}(\IV^2) - 1)$. Write $\sP^{d}_{r,nd} \subset \sP^{d}_{r}$ for the fibre of $\sP^{d}_{r} \to J^{d}_{r} \sS$ above $J^{d}_{r,nd} \sS$. Let $\che{L}_{v} \subset \che{L}$ denote the subscheme of flags whose $\Fil^1$ part contains the first vector $v$ in the standard basis of $R^m$.

Let $\Phi^d_r$ be the graph of the map $\alpha^d_r$ in \ref{bigjetthm} corresponding to $(\sH^2, \Fil^{\bullet}, \nabla)$. Define the intersections of schemes over $R[1/r!]$
\begin{align*}
\sI_{0} &:= \Phi^{d}_{0} \cap \left( \sP^d_r \times \che{L}_{v} \right) = \Phi^{d}_{0} \cap \left( \sP^{d}_{0} \times \che{L}_{v} \right) \text{ for $r = 0$} \\
\sI_{r} &:= \Phi^{d}_{r} \cap \left( \sP^{d}_{r,nd} \times (J^{d}_{r} \che{L}_{v} \setminus (\pi^{r}_{1})^{-1}(J^{d}_{1,c} \che{L}_{v})) \right) \text{ for $r > 0$}.
\end{align*}

\begin{remark}
    \label{rmk: unpack definition}
    For readers' convenience, let us unpack the definitions and give a description of $\sI_0$ and $\sI_{r}$ on field-valued points. Let $k$ be a field. Recall the description of $\sP^d_r(k)$ in \ref{sec: define P^d_r}. Then $\sI_0(k)$ corresponds to pairs $(p_0, \sigma_0) \in \sP^{d}_{0}(k) \times \che{L} (k)$ such that $p_0$ is given by a pair $(s, b^\bullet)$ where $s \in S(k)$ and $b^\bullet$ is a frame of $\sH^2_s$, $\psi_{b^\bullet}^{-1}(v) \in \Fil^1 \sH^2_s$, and the flag $\sigma_0$ is given by the filtration on $\sH^2_s$ via the choice of frame $b^\bullet$. An element of $\sI_r(k)$ is given by a pair $(p_r, \sigma_r) \in \sP^d_r(k) \times J^d_r \che{L}(k)$, such that $p_r$ is given by a triple $(s, b^\bullet, j_r : \ID^d_{r, k} \to \sS_k)$ where 
\begin{enumerate}[label=\upshape{(\roman*)}]
    \item $j_r$ is injective on tangent spaces, 
    \item the composition $\ID^d_{r, k} \stackrel{j_r}{\to} \Delta^r_s \stackrel{\psi_{b^\bullet}}{\to} \che{L}_k$ is nonconstant when restricted to $\ID^d_{1, k}$, 
    \item $(\psi_{b^\bullet} \circ j_r)^{-1} (v)$, as a global section of $j_r^* \sH^2$ over $\ID^d_{r, k}$, lies in the $\Fil^1$ part, 
\end{enumerate}
and $\sigma_r = \alpha^d_r(p_r) = \psi_{b^\bullet} \circ j_r$. For $r \ge r'$, there is a natural projection $\sI_{r} \to \sI_{r'}$ such that on $k$-points it is given by sending $(p_r := (s, b^\bullet, j_r), \alpha_{r}^d(p_r))$ to $(p_{r'} := (s, b^\bullet, j_{r'}), \alpha^d_{r'}(p_{r'}))$, where $j_{r'} = \pi^r_{r'}(j_r)$. 
\end{remark}

\begin{remark}
\label{rmk: I_r commutes with base change}
    Thanks to the compatibility-with-base-change statements in \ref{bigjetthm}, the formation of the tuple of $\sS$-schemes $\mathcal{T} := (\Phi^d_r, \sI_0, \sI_r)$ commutes with extension of scalars within $\IC$: If $R' \subseteq \IC$ is a sub-$R$-algebra, then the tuple $\mathcal{T}'$ of $\sS \tensor_R R'$-schemes formed by applying the above constructions to the family $(f : \sX \to \sS) \tensor_R R'$ is naturally identified with $\mathcal{T} \tensor_R R'$ in the obvious sense.  
\end{remark}

For later use, let us make some elementary observations on the topology of complex varieties. 
\begin{lemma}
\label{lem: infinite intersection}
Let $S$ be a $\IC$-variety and let $J_0 \supseteq J_1 \supseteq \cdots$ be an infinite sequence of descending constructible subsets of $S(\IC)$. Use $\cl(-)$ to denote the closure in the Zariski topology of $S(\IC)$.\footnote{We sometimes use implicitly the fact that for a constructible subset of $\IC$-points on a $\IC$-variety, analytic closure and Zariski closure coincide.} Then 
\begin{equation}
    \cl(\cap_{r \ge 0} J_r) = \cap_{r \ge 0} \cl(J_r)
\end{equation}
\end{lemma}
\begin{proof}
 Note that since $S(\IC)$ equipped with the Zariski topology is a Noetherian topological space, the infinite sequence of closed subsets $\cl(J_0) \supseteq \cl(J_1) \supseteq \cdots$ has to stabilize, i.e., for some $r_0$, $\cl(J_r) = \cl(J_{r_0})$ for all $r \ge r_0$, so that the right hand side is just $\cl(J_{r_0})$. 
 It is tautological that the left hand side is contained in the right hand side, so we only need to show the converse. It suffices to show that for every irreducible component $C$ of $\cl(J_{r_0})$, the intersection $(\cap_{r \ge 0} J_r) \cap C = \cap_{r \ge 0} (J_r \cap C)$ is Zariski dense in $C$. 
 
 Suppose now that this is not true. We divide into two cases. First, assume that for every $r$, $J_r \cap C$ contains a Zariski open subset $U_r$ of $C$. Then by the Baire category theorem (\cite[\href{https://stacks.math.columbia.edu/tag/0CQN}{Tag 0CQN}]{stacks-project}), the infinite intersection of opens $\cap_{r \ge 0} U_r$ is still analytically (and hence Zariski) dense in $C$, so we are done. Otherwise, by increasing $r_0$ if needed, $J_{r_0} \cap C$ does not contain an open subset of $C$. As $J_{r_0} \cap C$ is constructible, this means that it is contained in a strict (i.e., proper) closed subset $Z$ of $C$, but this contradicts the assumption that $C$ is a component of $\mathrm{cl}(J_{r_0})$. 
\end{proof}







\begin{corollary}
\label{cor: project infinite intersection}
    Let $f : S \to T$ be a morphism between $\IC$-varieties. Let $J_0 \supseteq J_1 \supseteq \cdots$ be a descending sequence of constructible subsets of $S(\IC)$. Then 
    \begin{equation}
        \cl(f(\cap_{r \ge 0} J_r)) = \cap_{r \ge 0} \cl(f(J_r))
    \end{equation}
\end{corollary}
\begin{proof}
    Let us start by recalling a tautology: By continuity, for every subset $A \subseteq S(\IC)$, we have $f(\cl(A)) \subseteq \cl(f(A))$, or equivalently $\cl(f(\cl(A))) = \cl(f(A))$. Now we note that for $r_0 \gg 0$, 
    \begin{align*}
        \cl(f (\cap_{r \ge 0} J_r) ) &\supseteq f(\cl(\cap_{r \ge 0} J_r)) \tag{\textrm{by continuity of $f$}} \\
        &= f(\cap_{r \ge 0} \cl(J_r)) \tag{\textrm{by \cref{lem: infinite intersection}}} \\
        &=  f(\cl(J_{r_0})) \tag{\textrm{by Noetherianity of $S$}} 
    \end{align*}
    As $\cl(f (\cap_{r \ge 0} J_r))$ is closed, we moreover have $\cl(f (\cap_{r \ge 0} J_r)) \supseteq \cl(f(\cl(J_{r_0}))) = \cl(f(J_{r_0}))$. On the other hand, by Noetherianity of $T(\IC)$, we have $\cap_{r \ge 0} \cl(f(J_r)) =  \cl(f(J_{r_0}))$ for $r_0 \gg 0$. Hence we obtain $\cl(f (\cap_{r \ge 0} J_r)) \supseteq \cap_{r \ge 0} \cl(f(J_r))$. The reverse inclusion is tautological: for every $r'$, $\cl(f(\cap_{r \ge 0} J_r)) \subseteq \cl(f(J_{r'}))$ and hence $\cl(f(\cap_{r \ge 0} J_r)) \subseteq \cap_{r'} \cl(f(J_{r'}))$. 
\end{proof}

Now we are ready to prove the following, which is a sharpening of the more abstract statement \cite[Prop. 5.8]{urbanik2022algebraic} to the situation which will be of interest to us.

\begin{proposition}
\label{Econstrprop}
There exists $r_{0} \in \IN_{> 0}$ such that if $E_{r_0} \subset \sS$ is the scheme-theoretic image of $\sI_{r_{0}} \to \sS$, then every irreducible component $E_{0} \subset E_{r_0, \IC}$ has dimension at least $d$ and the algebraic monodromy group of $\restr{\IV}{E_{0}}$ is properly contained in that of $\IV$.
\end{proposition}

The proof requires two notions from Hodge theory, which we now recall.

\begin{definition}
Given a variation of Hodge structure $\IV$ on a smooth complex algebraic variety $S$, we say an irreducible complex algebraic subvariety $Z \subset S$ is weakly special if it is maximal for its algebraic monodromy group. We say it is special if it is maximal for its Mumford-Tate group (the Mumford-Tate group of $\IV_s$ for a very general point $s \in Z(\IC)$).
\end{definition}

\begin{remark}
Special subvarieties correspond to Hodge loci for variations of Hodge structures in the tensor category generated by $\IV$. Since they are defined by rational tensors, there are countably many. On the other hand there are in general uncountably many weakly special subvarieties; for instance, if $\IV$ induces a quasi-finite period map, for every point $s \in S(\IC)$ the variety $\{ s \}$ is weakly special.
\end{remark}

We will also pass back and forth between complex analytic objects and formal objects, so we make the following convention: let $Y$ be a complex analytic space and $y \in Y$ be a point. By the formal completion of $Y$ at $y$ we mean the formal scheme $\Spf(\hat{\sO}^{\mathrm{an}}_{Y, y})$, where $\hat{\sO}^{\mathrm{an}}_{Y, y}$ is the formal completion of the germ of analytic functions on $Y$ at $y$. We shall denote it by $\hat{Y}_y$. A morphism $g : Y \to Y'$ between complex analytic spaces induces a morphism $\hat{Y}_y \to \hat{Y}'_{g(y)}$ between formal schemes.

\begin{proof}
If $d \leq 0$ then $\mathcal{I}_{r}$ is empty and there is nothing to show, so we may assume $d > 0$. \mpr{Let $E_{r}$ be the scheme-theoretic image of $\sI_{r} \to \sS$ and set $E = \cap_{r} E_{r}$. Note that the $E_{r}$ are closed subschemes and $\sS$ is Noetherian, which guarantees that $E = E_{r_0}$ for $r_0 \gg 0$. It therefore suffices to prove the proposition for $E$. }

\vspace{0.5em}

\noindent \textbf{The case $R = \IC$:} We first prove the statement when $R = \IC$. \mpr{Let $\sJ_r$ be the set-theoretic image of $\sI_r(\IC) \to (\sP \times \che{L}_{v})(\IC)$ and set $\sJ = \cap_{r \ge 0} \sJ_r$. Let $\pr_\sS$ denote the projection $\sP \times \che{L}_{v} \to \sS$. As the map $\sI_r \to \sS$ is nothing but the composition of $\sI_r \to \sP \times \che{L}_{v}$ with $\pr_\sS$, the Zariski closure of each $\pr_\sS(\sJ_r)\subseteq \sS(\IC)$ is just $E_{r}(\IC)$. As $\{ \sJ_r \}$ forms a descending chain of constructible subsets of $(\sP \times \che{L}_{v})(\IC)$, by \ref{cor: project infinite intersection} we have 
\begin{equation}
\label{eqn: prS J dense in E}
    \cl(\pr_{\sS}(\sJ)) = \cl(\pr_{\sS}(\cap_{r \ge 0} \sJ_r)) = \cap_{r \ge 0} \cl(\pr_\sS(\sJ_r)) = \cap_{r \ge 0} E_r(\IC) = E(\IC).  
\end{equation} In particular, $\pr_\sS(\sJ)$ is Zariski dense in $E$, and analytically dense in $E(\IC)$.}

Let $(p, \sigma)$ be a $\IC$-point in $\sJ$. A priori being in $\sJ$ means that there exists a sequence of pairs $\{ (p_{r}, \sigma_{r}) \in \sI_r \}_{r \geq 0}$ such that $(p, \sigma)$ is the image of $(p_{r}, \sigma_{r})$ under the natural projection $\sP^d_r \times J^d_r \che{L}_{v} \to \sP \times \che{L}_{v}$. However, applying \cite[Lem.~5.3]{urbanik2021transcendence}\footnote{Note that the letter ``$E$'' in \textit{loc. cit.} differs from our $E$ in the present context.} to the sequence $\{ \sI_{r} \}_{r \geq 0}$ at the point $(p, \sigma)$, we can arrange the sequence $\{ (p_{r}, \sigma_{r}) \}_{r \geq 0}$ so that for each $r > 0$, $(p_{r - 1}, \sigma_{r - 1})$ is the image of $(p_{r}, \sigma_{r})$ under the natural projection $\sP^d_r \times J^d_r \che{L}_{v} \to \sP^d_{r - 1} \times J^d_{r - 1} \che{L}_{v}$. Let $\ID^d_{\infty, \IC} := \varinjlim_r \ID^d_r$, which is (abstractly) isomorphic to $\Spf(\IC[\![x_1, \cdots, x_d]\!])$. The sequence $\{ (p_{r}, \sigma_{r}) \}_{r \geq 0}$ then defines formal objects $\sigma_{\infty} : \ID^{d}_{r,\IC} \to \che{L}_{v,\IC}$, $\psi_{\infty} : \hat{\sS}_{s} \to \che{L}_{\IC}$ and $j_{\infty} : \ID^{d}_{\infty,\IC} \to \hat{\sS}_{s}$ satisfying $\sigma_{\infty} = \psi_{\infty} \circ j_{\infty}$. By \ref{rmk: unpack definition}, $j_\infty$ is injective on tangent spaces, which implies that the map $j_{\infty}$ is a closed embedding, and $\sigma_{\infty}$ is non-constant.


Write $(\hat{\sH}^2_{s}, \hat{\Fil}^{\bullet}_{s}, \hat{\nabla}_{s})$ for the restriction of the data $(\sH^2, \Fil^{\bullet}, \nabla)$ to $\hat{\sS}_{s}$. Let $b^{1}, \hdots, b^{m} \in \sH_s$ be vectors in $b^\bullet$, which we shall extend to flat sections which trivialize $\hat{\sH}^2_{s}$. Choose $A \in \mathrm{GL}_m(\IC)$ so that the $A$-translated basis $\{ b^i_A := A \cdot b^i \}$ is identified under the Betti-de Rham comparison with an integral basis for $\IV^2_s$, and hence $\psi_{A} := A \cdot \psi_{\infty}$ is the completion of an honest Hodge-theoretic local period map (to be explained in more detail in the paragraph below). Transfer the polarization on $\mathbb{V}^2_s$ along the isomorphism $i_{s} : \IV^{2}_{s} \cong \IQ^{m}$ induced by $b^{1}_{A}, \hdots, b^{m}_{A}$ to obtain a polarization $Q : \IQ^{m} \otimes \IQ^{m} \to \IQ$ and define $\che{D} \subset \che{L}$ to be the subvariety consisting of those Hodge flags satisfying the first Hodge-Riemann bilinear relation. Then the local period map $\psi_A$ factors through $\che{D}$, and by the large monodromy assumption and elementary linear algebra (c.f. \cite[Ch.2, pg.48]{zbMATH06031035}) $\che{D}$ is exactly the orbit of the algebraic monodromy group of $\IV$ when regarded as a subgroup of $\GL_{m}(\IC)$ via $i_{s}$. Note here that $\che{L}$ and $\che{D}$ are flag varieties parameterizing flags with the same Hodge numbers as $\IV^{2}$ rather than $\IV$, otherwise we would have to include conditions coming from global flat sections as well. Recall however that $\IV$ and $\IV^2$ have the same algebraic monodromy group.

We now relate our construction to a Hodge-theoretic situation so that we may apply the Ax-Schanuel Theorem \cite{zbMATH07066495} of Bakker-Tsimerman. Below we shall view $\sS$ as a complex analytic space, and let $B \subseteq \sS$ be a sufficiently small complex analytic open ball containing $s$. The basis $b^i_A$ of $\sH_s$ defines a trivialization of the restriction of (the analytification of) $\sH$ over $B$, which induces a period morphism $\psi^{\mathrm{an}}_A : B \to \che{L}_\IC$ in the category of the complex-analytic spaces. When we identify the formal completion of $B$ at $s$ with $\hat{\sS}_s$, $\psi_A$ is nothing but the completion of $\psi^{\mathrm{an}}_A$ at $s$.




Define $\che{D}_{v} = \che{D} \cap (A \cdot \che{L}_{v})$. Then if we consider a component $U$ of the intersection $(\sS \times \che{D}_{v}) \cap \Gamma_{A}$ inside $\sS \times \che{D}$, where $\Gamma_{A}$ denotes the graph of $\psi^{\mathrm{an}}_{A}$, the Ax-Schanuel Theorem \cite{zbMATH07066495} tells us one of the following two situations happens:
\begin{itemize}
\item[(a)] we either have that
\begin{align*}
 \codim_{\sS \times \che{D}} \, U &\geq \codim_{\sS \times \che{D}} (\sS \times \che{D}_{v}) + \codim_{\sS \times \che{D}} \, \Gamma_{A}
 \end{align*}
 or equivalently that
\begin{align*}
\codim_{\sS} \, \textrm{pr}_{\sS}(U) &\geq \codim_{\che{D}} \che{D}_{v},
\end{align*}
where $\pr_\sS(U)$ is the projection of $U$ to $\sS$;
or
\item[(b)] $\pr_\sS(U)$ lies inside a strict (i.e., properly contained in $\sS$) weakly special subvariety $Z$ of $\sS$ (c.f. \cite[Cor. 4.14]{zbMATH07387423} for an equivalence with the notion of weakly special used in \cite{zbMATH07066495}).
\end{itemize}
Note that the projection $\Gamma_A \to B$ is an isomorphism, so that $U$ maps isomorphically onto $\pr_\sS(U)$, which is in particular an analytic subset of $B$.

We may choose $U$ so that $j_{\infty}$ factors through $\hat{\pr}_\sS(U)$, the formal completion of $\pr_\sS(U)$ at $s$. Now by Lemma \ref{h20estimatelem} proven below, one has $\codim_{\che{D}} \che{D}_{v} = h^{2,0}$, and because $j_{\infty}$ embeds a formal disk of dimension $d = \dim_{\IC} \sS - h^{2,0} + 1$ into $\hat{\pr}_{\sS}(U)$ we must necessarily be in case (b). Moreover, since $\psi_{A} \circ j_{\infty}$ is non-constant by construction, $\psi_{A}$ is non-constant along the formal completion of $Z$ at $s$.

This implies that up to shrinking $B$, $\psi_{A}^{\mathrm{an}}$ is non-constant on $B \cap Z^{\mathrm{sm}}$, where $Z^{\textrm{sm}}$ is the smooth locus $Z$. It follows that $Z^{\textrm{sm}}$ admits a non-constant period map. We note also that each point in $\textrm{pr}_{\sS}(U)^{\textrm{sm}} \subset (\psi_{A}^{\mathrm{an}})^{-1}(\che{D}_{v})$ lies in the image of $\mathcal{J}$ by construction: the locus $\textrm{pr}_{\sS}(U)$ is closed and irreducible in $B$ of dimension $\geq d$, so for each $r \ge 1$ we may apply \ref{rmk: unpack definition} with the $r$-limps associated to $\restr{\psi_{A}^{\mathrm{an}}}{B}$ and jets $j : \mathbb{D}^d_{r,\mathbb{C}} \to \textrm{pr}_{\sS}(U)^{\textrm{sm}}$.\footnote{Technically, we are mapping $\mathbb{D}^d_{r,\mathbb{C}}$ into the formal completions of $\textrm{pr}_{\sS}(U)^{\textrm{sm}}$ at points in it.} This implies that $\textrm{pr}_{\sS}(U)$ is contained in $E \cap Z$, so that $\dim E \cap Z \ge d$. Thus we have shown: 

\begin{quote}
\textbf{Claim:} Any point $(p, \sigma) \in \mathcal{J}$ projects into a strict weakly special subvariety $Z \subset \sS$ such that $\dim_{s} (Z \cap E) \geq d$, where $p$ lies above $s \in \sS$ ($\dim_s$ denotes the dimension at $s$). 
\end{quote}

We now show the statement of the Proposition. It is a consequence of the Theorem of the Fixed part that the period map on $Z^{\textrm{sm}}$ being non-constant implies the algebraic monodromy group of $\restr{\IV}{Z^{\textrm{sm}}}$, which agrees with the algebraic monodromy group of $\restr{\IV}{Z}$, is non-trivial. It now follows from the $\IQ$-simplicity of the monodromy on $\sS$ that every such $Z$ is contained in a strict weakly non-factor weakly special subvariety $Z'$ in the sense of \cite[Def. 1.13]{FieldsDefHgLoci}. Because the monodromy group is $\mathbb{Q}$-simple, by \cite[Lem 2.5]{FieldsDefHgLoci} $Z'$ is special. Since there are countably many special subvarieties, we conclude there is a countable collection $\{ Z_{i} \}_{i \in I}$ of strict closed subvarieties $Z_{i} \subset \sS$ such that any strict weakly special subvariety with non-trivial algebraic monodromy group lies inside one of the $Z_{i}$. In particular, $\textrm{pr}_{\sS}(\mathcal{J}) \subset \cup_{i \in I} Z_{i}$ as a consequence of the claim.

\mpr{Now we argue that $\cl(\pr_\sS (\sJ)) = E(\IC)$ in fact lies in finitely many $Z_i$'s. For each $i$, set $\wt{Z}_i = \pr_\sS^{-1}(Z_i)$. Then $\sJ$ is contained in the union $\cup_i \wt{Z}_i$. It suffices to argue that $\sJ$ is contained in $\cup_{i \in I'} \wt{Z}_{i}$ for some finite $I' \subseteq I$, because then $\cup_{i \in I'} Z_i$ is closed and contains $\pr_\sS(\sJ)$, and hence its Zariski closure $\cl(\pr_\sS(\sJ))$. Now, let $C$ be an irreducible component of $\cl(\sJ)$. We reduce to showing that $C(\IC)$ is contained in one of the $\wt{Z}_i$'s. As $\sJ$ is an infinite intersection of constructible subsets of $(\sP \times \che{L}_v)(\IC)$, the complement of $(\sJ \cap C)(\IC)$ in $C(\IC)$ is contained in a countable union $\cup_{j \in J} T_j$ of closed subvarieties. Moreover, note that each $T_j$ may be taken to be a \textit{strict} closed subvariety because $(\sJ \cap C)(\IC)$ is dense in $C(\IC)$. Now we observe that $C(\IC)$ is contained in the countable union of closed subsets $(\cup_{i \in I} (\wt{Z}_i \cap C)(\IC)) \cup (\cup_{j \in J} T_j(\IC))$. By the Baire category theorem \cite[\href{https://stacks.math.columbia.edu/tag/0CQN}{Tag 0CQN}]{stacks-project}, not all of these closed subvarieties can be strictly contained in $C(\IC)$. This implies that one of them, which has to be one of the $(\wt{Z}_i \cap C)(\IC)$'s, coincides with $C(\IC)$. This proves the proposition for $R = \IC$.  }

\vspace{0.5em}


\noindent \textbf{General $R$:} Returning now to our original problem, we choose $r_{0}$ such that $E = E_{r_{0}}$. By \ref{rmk: I_r commutes with base change} and the fact that scheme-theoretic image commutes with flat base change \cite[Tag~081I]{stacks-project}, the base change $E_\IC$ of $E$ to $\IC$ can be viewed as the ``$E$'' we discussed above in the $R = \IC$ case. The result follows. \end{proof}

\begin{lemma}
\label{h20estimatelem}
Suppose that $\IZ^m$ is equipped with a pure Hodge structure of weight $2$ and a polarization $Q$. Let $\overline{h} = (h^{2,0}, h^{1,1}, h^{0,2})$ be the Hodge numbers and $\che{D}$ be the variety of all flags on $\IC^{m}$ with Hodge numbers $\overline{h}$ satisfying the first Hodge-Riemann bilinear relation, and let $v \in \IC^{m}$ be any vector for which there exists $F^{\bullet} \in \che{D}(\IC)$ such that $v \in F^{1}$. Then the locus
\[ \che{D}_{v} = \{ F^{\bullet} \in \che{D}(\IC) : v \in F^{1} \} \]
has codimension $h^{2,0}$ in $\che{D}$.
\end{lemma}

\begin{proof}
The complex points of the group $G = \textrm{Aut}(\IZ^{m}, Q)$ act transitively on $\che{D}$, so we may reduce to the case where $v \in \IQ^{m}$ is a Hodge vector. The result then follows from \cite[Prop 5.4]{Carlson2009}.
\end{proof}

\subsection{The Power-primitive Case} We continue to work with the family $f: \sX \to \sS$ in \S\ref{sec: general setup}, but with the base ring $R$ set to be $\sO_L[1/N]$ for some number field $L \subseteq \IC$ and $N \in \IN_{> 0}$. We shall assume in addition assume that $\sO_L[1/N]$ is unramified over $\IZ[1/N]$. This implies that if $s \to \sS$ is an alg-geometric point and $k := k(s)$ has characteristic $p > 0$, then the morphism $\sO_{L}[1/N] \to k$ determined by $s$ uniquely lifts to a morphism $\sO_{L}[1/N] \to W(k)$, so that we may form $\sS \tensor_{\sO_{L}[1/N]} W(k)$. Below we shall simply write it as $\sS \tensor W(k)$.

\begin{lemma}
\label{lem: global LB def}
    Let $(s, \xi)$ be a LB-pair for $\sX/\sS$ such that $\mathrm{char\,} k = p > 0$ for $k := k(s)$. Then the following statements are equivalent. 
    \begin{enumerate}[label=\upshape{(\roman*)}]
        \item The class of $\xi$ in $\NS(\sX_s) \tensor \IZ_{(p)}$ lies in $\NS(\sX_\eta)_s \tensor \IZ_{(p)}$ (recall \ref{def: power-primitive}). 
        \item Let $\hat{\sS}_s$ be the formal completion of $\sS \tensor W(k)$ at $s$ and $\Def(\xi) \subseteq \hat{\sS}_s$ be as defined in \ref{thm: Deligne generalized} for the restriction of $\sX$ to $\hat{\sS}_s$. Then $\Def(\xi) = \hat{\sS}_s$. 
        \item $(s, \xi)$ is global on $\sS$ in the sense of \ref{def: localizations}, i.e., $\shP(\sS, s, \xi)$ dominates $\sS$. 
    \end{enumerate}
\end{lemma}
\begin{proof}
    (i) $\Rightarrow$ (ii): By \ref{lem: Def(xi)} we may replace $\xi$ by a prime-to-$p$ power and assume that the class of $\xi$ lies in $\NS(\sX_\eta)_s$. Moreover, as $\Def(\xi)$ depends only on the class of $\xi$ in $\NS(\sX_s)$, we may assume that $\xi$ is global on $\sS \tensor W$ by \ref{lem: generic LB}(a). Then we apply \ref{globalimagelem2}. 

    (ii) $\Rightarrow$ (iii): This is clear, because by \ref{thm: Deligne generalized}, (ii) implies that the formal completion of $\shP_{\sX/\sS}$ at $\xi$ is formally smooth over $\hat{\sS}_s$ and hence $\shP(\sS, s, \xi)$ dominates $\sS$. 

    (iii) $\Rightarrow$ (i): That $\shP(\sS, s, \xi)$ dominates $\sS$ implies that the class of $\xi$ in $\NS(\sX_s)$ lies in the image of $\NS(\sX_{\bar{\eta}})$ under the specialization map $\mathrm{sp}_\gamma : \NS(\sX_{\bar{\eta}}) \to \NS(\sX_s)$ for some path $\gamma$ connecting $\bar{\eta}$ and $s$ \ref{sec: specialization}. However, we have $\NS(\sX_\eta) \tensor \IZ_{(p)} = \NS(\sX_{\bar{\eta}}) \tensor \IZ_{(p)}$ by assumption (v) in the setup of \S\ref{sec: general setup}. 
\end{proof}

\begin{lemma}
\label{multicharcodimlem}
Let $E \subset \sS$ and $r_0 \in \IN$ be as introduced in $\ref{Econstrprop}$ and assume that $r_0!$ is invertible in $R$. Let $\sE \subseteq \sS$ be the union of $E$ and the smallest closed subscheme $E'$ containg all points $t \in \sS$ such that the Kodaira-Spencer map has rank $< h^{0, 2}(\IV) = \mathrm{rank\,} \sH/ \Fil^1 \sH$.  

Let $(s, \xi)$ be an LB-pair over $\sS \smallsetminus \sE$ and set $k := k(s)$. Let $\hat{\sS}_{s, k}$ be the formal completion of $\sS_{k}$ at $s$, and let $\Def(\xi)_k$ be the scheme-theoretic image in $\hat{\sS}_{s, k}$ of the formal completion $\hat{\shP}^\xi_k$ of $\shP_{\sX/\sS} \tensor k$ at $(s, \xi)$. If $\xi$ is power-primitive, then $\Def(\xi)_k$ has codimension $h^{0,2} := h^{0, 2}(\IV)$ in $\sS_{k}$.
\end{lemma}

\begin{proof}
If $\textrm{char}\, k = p > 0$ the fact that $\Def(\xi)_k$ has at most codimension $h^{0,2}$ follows from \ref{thm: Deligne generalized}. If $\textrm{char}\, k = 0$ the same fact follows from \ref{sec: Def(xi) in char 0}. We are then reduced to proving the contrapositive statement which shows that if for an LB-pair $(s, \xi)$ the locus $\Def(\xi)_k$ has codimension strictly less than $h^{0,2}$ in $\sS_{k}$ then $(s, \xi)$ lies over $\sE$. 

Since the map $\hat{\shP}^\xi_k \to \Def(\xi)_k$ is formally smooth with fiber dimension $h^{0, 1}$, the condition that $\Def(\xi)_k$ has codimension less than $h^{0,2}$ is equivalent to the condition that there is a component $C$ of $(\shP_{\sX/\sS} \tensor k)_\red$ containing $(s, \xi)$ which has dimension greater than $m - h^{0, 2} + h^{0, 1}$, where $m$ is the dimension of $\sS_{k}$ over $k$. In particular, the scheme-theoretic image $Z'$ of $C$ on $\sS_k$ has dimension $> m - h^{0, 2}$. To show that $(s, \xi)$ lies over $\sE$, it will suffice to show that $Z'$ lies inside $\sE$. More concretely, we may move the pair $(s, \xi)$ along $C$ and assume that 
\begin{enumerate}[label=(\alph*)]
    \item $s$ lies on the smooth locus of $Z'$, and 
    \item $(s, \xi)^\flat$ lies on $C$ but not on any other irreducible component of $\shP_{\sX/\sS} \otimes k$.
\end{enumerate}
Because such $s$ are Zariski dense in $Z'$, we may reduce to showing the claim for LB-pairs $(s, \xi)$ satisfying these conditions. (We note here that being power-primitive is an open condition on $\shP_{\sX/\sS}$.) Note also that $\Def(\xi)_{k,\red}$ automatically lies inside $Z'$, and is therefore isomorphic to $\Spf(k[\![x_{1}, \hdots, x_{e}]\!])$ for $e = \dim Z' > m - h^{0, 2}$. Fix such a choice of coordinates. 

We shall now assume that $s \not\in E'_k$ and show that $s \in E_k$, which will complete the proof. As $E$ is defined to be the scheme-theoretic image of $\sI_{r_0} \to \sS$, it suffices to construct a preimage of $s$ in $\sI_{r_0}(k)$. Recall the description of the field-valued points of $\sI_{r_0}$ in \ref{rmk: unpack definition}. We may choose an isomorphism $\hat{\sS}_{s,k} \cong \Spf(k[\![x_1, \cdots, x_{m}]\!])$ such that the coordinates agree with the chosen ones of $\Def(\xi)_{k,\red}$ \mpr{(that is, the first $e$ coordinates are those of $\Def(\xi)_{k,\red}$ that we fixed)}. Let $j : \ID^d_{r_0, k} \to \sS_k$ be the jet centered at $s$ defined by $k[\![x_1, \cdots, x_d]\!]/(x_1, \cdots, x_d)^{r_0 + 1}$ for these chosen coordinates on $\Def(\xi)_{k,\red}$. \mpr{(Here $d = \dim_{R} \mathcal{S} - (h^{0,2}(\mathbb{V}^2) - 1) = m - h^{0, 2} + 1$, as before.)} \mpr{As $d \le e$, $j$ maps into $\Def(\xi)_k$. As $j$ is by construction a closed embedding, and in particular injective on tangent spaces, \ref{rmk: unpack definition}(i) is satisfied.} Next, we choose a basis $b^\bullet := \{ b^i \}$ of $\sH^2_{s}$ such that $b^1 = (\pi_2 \circ c_{1, \dR})(\xi)$ (recall that $\pi_2$ is the projection $\sH \to \sH^2$). If $\textrm{char}\, k = p$ this is allowed because $\xi$ is power-primitive, so that by \ref{lem: primitive line bundle} $(\pi_2 \circ c_{1, \dR})(\xi) \in \sH^2_s$ is nonzero. Now consider the $r_0$-limp $\psi_{b^\bullet} : \Delta^{r_0}_s \to \che{L}_k$ (cf. \ref{rlimpdef}). 

We have a commutative diagram (cf. (\ref{eqn: diag for KS}))
\[\begin{tikzcd}
	{T_s \sS_k} & { \Hom(\Fil^1 \sH^2_s / \Fil^2 \sH^2_s, \sH^2_s / \Fil^1 \sH^2_s) } \\
	{T_s \Delta^{r_0}_s} & {T_{\psi_{b^\bullet}(s)}\che{L}_k}
	\arrow[from=1-1, to=1-2]
	\arrow[equal, from=1-1, to=2-1]
	\arrow["{T_s(\psi_{b^\bullet})}"', from=2-1, to=2-2]
	\arrow[hook, from=1-2, to=2-2].
\end{tikzcd}\]
By the assumption that $s \not\in E'_k$, ${T_s(\psi_{b^\bullet})}$ has rank $\ge h^{0, 2}$, and hence cannot vanish on the tangent space of $\Def(\xi)_{k,\red}$, or that of $\ID^d_{r_0, k}$. This verifies \ref{rmk: unpack definition}(ii). Finally, as $\xi$ deforms over $\ID^d_{r_0, k} \stackrel{j}{\into} \Def(\xi)_k$, $b^1 = (\pi_2 \circ c_{1, \dR})(\xi)$ extends to a flat section of $\Fil^1 \sH_{k}$ over $\ID^d_{r_0, k}$. This verifies \ref{rmk: unpack definition}(iii). Therefore, $((s, b^\bullet, j), \psi_{b^\bullet} \circ j)$ defines a point in $\sI_{r_0}(k)$ lying above $s$, as desired. 
\end{proof}

\begin{corollary}
\label{cor: power-primitive case}
Suppose that in the situation of \ref{multicharcodimlem} $\mathrm{char\,} k = p > 0$. Let $\hat{\sS}_s$ be the formal completion of $\sS_W := \sS \tensor W$ at $s$ for $W := W(k)$, and let $\Def(\xi)$ be the scheme-theoretic image \mpr{in $\hat{\sS}_s$} of the formal completion $\hat{\shP}^\xi$ of $\shP_{\sX/\sS} \tensor W$ at $(s, \xi)$. If $\xi$ is power-primitive, then $\Def(\xi)$ is flat over $W$. 
\end{corollary}
\begin{proof}
By \ref{thm: Deligne generalized}, $\Def(\xi) \subseteq \hat{\sS}_s$ is cut out by at most $h^{0, 2}$ equations, and by \ref{multicharcodimlem}, $\Def(\xi)_k = \Def(\xi) \tensor k$ has codimension exactly $h^{0, 2}$ in $\hat{\sS}_s \tensor k$. This implies that $\Def(\xi)$ has codimension exactly $h^{0, 2}$ in $\hat{\sS}_s$ and hence is Cohen-Macaulay by \cite[02JN]{stacks-project}. Moroever, by miracle flatness \cite[00R4]{stacks-project} $\Def(\xi)$ is flat. 
\end{proof}

\begin{remark}
    \label{rmk: primitive implies general} It will be clear from \ref{sec: actual proof of main thms} below and the proof therein that a fortiori the conclusion of \ref{multicharcodimlem} and \ref{cor: power-primitive case} continue to hold if the assumption that $\xi$ is power-primitive is replaced by the more general assumption that $(s, \xi)$ is not global over $\sS$. However, due to the nature of the jet theoretic arguments, treating the power-primitive case is an important step in treating this more general case. 
\end{remark}

\subsection{Proofs of Main Theorems} Let $\sS$ be as in the introduction. Namely, let $L \subseteq \IC$ be a number field, $N \in \IN_{\geq 1}$ and $\sS$ be a smooth quasi-projective $\sO_L[1/N]$-scheme such that $\sS_\IC$ is connected. Let $f : \sX \to \sS$ be a smooth projective morphism with geometrically connected fibers. Assume that $f_\IC$ has large monodromy and sufficiently large period image. Note that to prove \ref{mainthm1} and \ref{mainthm2}, we are allowed to increase $N$, so that we may also assume that the family is admissible \ref{lem: admissible for large primes}. Moreover, suppose that there is a finite extension $L'$ of $L$ in $\IC$, $N' \ge N$, and a smooth scheme $\sS'$ over $\sO_{L'}[1/N']$ such that the generic fiber $\sS'_{L'}$ remains geometrically connected (as an $L'$-variety), and $\sS'$ is an \'etale cover of $\sS \tensor_{\sO_L[1/N]} \sO_{L'}[1/N']$. Then to prove the theorems we may replace $(L, N, \sS)$ by $(L', N', \sS')$ and $\sX$ by its restriction to $\sS'$. 

\begin{lemma}
\label{lem: replacement}
    After such a replacement, we may assume that the family $f : \sX \to \sS$ satisfies the assumptions in \S\ref{sec: general setup} with $R = \sO_L[1/N]$. 
\end{lemma}
\begin{proof}
    Only conditions (ii) and (v) \S\ref{sec: general setup} need explanation. The other conditions are automatically satisfied up to increasing $N$. To satisfy (ii), we first find some connected \'etale cover $\sS'_\IC$ of $\sS_\IC$ such that $\Mon(R^2 f_{\IC*} \IQ|_{\sS'_\IC})$ is connected. Let $\overline{L} \subseteq \IC$ be an algebraic closure of $L$. Since for any $s \in \sS(\IC)$, $\pi_1^\et(\sS_{\overline{L}}, s) = \pi_1^\et(\sS_\IC, s)$, such $\sS'_\IC$ is defined over $\overline{L}$ and hence over some finite extension $L'$ of $L$ in $\IC$. Then by spreading out, we obtain some $N' > N$ and $\sO_{L'}[1/N']$-scheme $\sS'$ admittible an \'etale $\sO_L[1/N']$-morphism to $\sS$. Make replacements accordingly. Then to satisfy (v), we note that the action of $\Gal_{k(\eta)} = \Gal(\eta/\bar{\eta})$ on $\NS(\sX_{\bar{\eta}})_{\mathrm{tf}}$ factors through a finite quotient and the surjective open morphism $\Gal_{k(\eta)} \to \pi_1^\et(\sS, \bar{\eta})$. Hence some finite open subgroup of $\pi_1^\et(\sS, \bar{\eta})$ acts trivially on  $\NS(\sX_{\bar{\eta}})_{\mathrm{tf}}$. Hence by replacing $\sS$ by a connected \'etale cover, $\Gal_{k(\eta)}$ acts trivially on $\NS(\sX_{\bar{\eta}})_{\mathrm{tf}}$. Then we conclude by \ref{lem: Galois descent}. 
\end{proof}

\begin{lemma}
\label{lem: pedantic lemma for lci}
    Let $k$ be a field and $S$ be a local $k$-algebra essentially of finite type with residue field $k$. Then $S$ is a complete intersection (over $k$) in the sense of \cite[00SD]{stacks-project} if and only if its completion $\hat{S}$ is of the form $k[\![x_1, \cdots, x_n]\!]/I$ for some $n$ and ideal $I$ generated by a regular sequence. 
\end{lemma}
\begin{proof}
    This can be deduced from \cite[09Q1]{stacks-project}, since by Cohen structure theorem \cite[032A]{stacks-project} there exists a surjection $k[\![x_1, \cdots, x_n]\!] \twoheadrightarrow \hat{S}$ for some $n$. 
\end{proof}

\subsubsection{} \label{sec: actual proof of main thms} Now we shall prove \ref{mainthm1} and \ref{mainthm2} under the assumption that $\sX/\sS$ satisfies the assumptions in \S\ref{sec: general setup}. Moreover, we may assume that $\sO_L[1/N]$ is unramified over $\IZ[1/N]$. Let $k$ be an algebraically closed field of characteristic $p > 0$ and write $W$ for $W(k)$. Let $(s, \xi)$ be an LB-pair defined over $k$ and $\hat{\sS}_s$ be the formal completion of $\sS_W := \sS \tensor W$ at $s$. Let $\sE$ be the proper closed subscheme of $\sS$ defined in \ref{multicharcodimlem}. Assume that $s \not\in \sE_k$. Let $\Def(\xi)$ be the deformation space of $\xi$ when $\sX$ is restricted to $\hat{\sS}_s$ (cf. \ref{thm: Deligne generalized}). 

For \ref{mainthm1} what we need to show is that $\Def(\xi)$ admits a point valued in $\sO_K$ for some finite extension $K$ over $K_0 := W[1/p]$. By \ref{lem: global LB def} it suffices to consider the case when the class of $\xi$ in $\NS(\sX_s) \tensor \IZ_{(p)}$ does not lie in $\NS(\sX_\eta)_s \tensor \IZ_{(p)}$, as otherwise $\Def(\xi) = \hat{\sS}_s$.  As $\Def(\xi)$ depends only on the class of $\xi$ in $\NS(\sX_s)$, using \ref{lem: primitive line bundle} we may assume that $m' \xi = \xi_0 + p^h \xi'$ for some $\xi_0$ whose class lies in $\NS(\sX_\eta)_s$, a power-primitive $\xi'$, $h \in \IN_{> 1}$ and $m' \in \IN \smallsetminus p\IN$. By \ref{lem: Def(xi)} we are free to replace $\xi$ by a prime-to-$p$ power, so we may assume $m' = 1$. By \ref{globalimagelem2}, $\xi_0$ deforms to the entire $\hat{\sS}_s$, so that $\Def(\xi) = \Def(p^h \xi')$. Clearly, $\Def(p^h \xi')$ contains $\Def(\xi')$, so it suffices to show that $\Def(\xi')$ admits some $\sO_K$-point for some $K$. Then the conclusion follows from \ref{cor: power-primitive case}, which tells us that $\Def(\xi')$ is flat over $W$, and hence we can find an $\sO_K$-valued point for some finite extension $K$ of $K_0$ by \ref{lem: Deligne's little lemma}. 

Now we use \ref{mainthm1} to prove \ref{mainthm2}. Let $(s, \xi) \to \shP_{\sX/\sS}$ be any LB-pair as in the first paragraph. Our goal is to show the following two statements. 
\begin{enumerate}[label = \upshape{(\alph*)}]
    \item The relative Picard scheme $\shP_{\sX/\sS}$ is syntomic over $\sO_L[1/N]$ at the point $(s, \xi)^\flat$. 
    \item Let $C$ be any irreducible component of $(\shP_{\sX/\sS})_\red$ which does not dominate $\sS$. Then the scheme-theoretic image $Z$ of $C$ on $\sS$ is flat over $\sO_L[1/N]$ and has codimension exactly $h^{0, 2} := \mathrm{rank\,} \sH/ \Fil^1 \sH$. 
\end{enumerate}

Again, if the class of $\xi$ in $\NS(\sX_s) \tensor \IZ_{(p)}$ lies in $\NS(\sX_\eta)_s \tensor \IZ_{(p)}$, then by \ref{lem: global LB def} $\Def(\xi) = \hat{\sS}_s$. As $\hat{\shP}^\xi$ is formally smooth over $\Def(\xi)$, we have in this case that $\shP_{\sX/\sS}$ is smooth over $\sS$ at $(s, \xi)^\flat$. Therefore, we again restrict to considering the case when the class of $\xi$ does not lie in $\NS(\sX_\eta)_s \tensor \IZ_{(p)}$. Note that \ref{lem: global LB def} tells us that in this case $(s, \xi)$ is not global over $\sS$, so that none of the irreducible components of $(\shP_{\sX/\sS})_\red$ which contains $(s, \xi)^\flat$ dominates $\sS$. 

We first argue that, in order to show (a), it suffices to show that $\Def(\xi) \tensor k \subseteq \hat{\sS}_s \tensor k$ has codimension exactly $h^{0, 2}$. Recall that we used $\shP(\sS, s, \xi)$ to denote the localization of $\shP_{\sX/\sS}$ at $(s, \xi)^\flat$ \ref{def: localizations}. Let $\fp \in \Spec(\sO_L[1/N])$ be the image of $(s, \xi)^\flat \in \shP_{\sX/\sS}$. By \cite[01UE]{stacks-project}, to show that $\shP_{\sX/\sS}$ is syntomic over $\Spec(\sO_L[1/N])$ at $(s, \xi)^\flat$ it suffices to show the following statements.

\begin{enumerate}[label=\upshape{(\roman*)}]
    \item $\shP(\sS, s, \xi)$ is flat over the localization 
     $\sO_{L, \fp}$. 
    \item $\shP(\sS,s, \xi) \tensor_{\sO_L} k(\fp)$ is a complete intersection over $k(\fp)$. 
\end{enumerate}
Note that $s$ induces an embedding $k(\fp) \into k$. We reduce to showing 
\begin{enumerate}[label=\upshape{(\roman*)}]
    \item[(iii)] $\hat{\shP}^\xi$ is flat over $W$, and 
    \item[(iv)] $\hat{\shP}^\xi_k := \hat{\shP}^\xi \tensor_W k$ is a complete intersection over $k$. 
\end{enumerate}
That (iii) implies (i) is clear, as $\hat{\shP}^\xi$ is faithfully flat over $\shP(\sS, s, \xi)$. That (iv) implies (ii) follows from \cite[00SI]{stacks-project} and \ref{lem: pedantic lemma for lci}. Now, using the fact that $\hat{\shP}^\xi$ is formally smooth over $\Def(\xi)$, in (iii) and (iv) we may equivalently replace $\hat{\shP}^\xi$ by $\Def(\xi)$. Then both statements would follow if $\Def(\xi)_k = \Def(\xi) \tensor k$ has codimension exactly $h^{0, 2}$ in $\hat{\sS}_{s, k}$, as by \ref{thm: Deligne generalized} $\Def(\xi) \subseteq \hat{\sS}_s$ is cut out by at most $h^{0, 2}$ equations (cf. the argument for \ref{cor: power-primitive case}).

Next, we argue that (b) implies (a). Indeed, suppose for the sake of contradiction that (b) holds but $\Def(\xi) \tensor k \subseteq \hat{\sS}_s \tensor k$ has codimension $< h^{0, 2}$. Then we have $\dim \Def(\xi) \tensor k \ge m - h^{0, 2} + 1$. Here $m$ denotes the relative dimension of $\sS$ over $\sO_L[1/N]$. Recall that $\sX \to \sS$ is an admissible family, so that $h^{0, 1}$ of its fibers are constant. As $\hat{\shP}^\xi$ is formally smooth over $\Def(\xi)$ with fiber dimension $h^{0, 1}$, there exists an irreducible component $C'$ of $(\shP_{\sX/\sS} \tensor_{\sO_L} k(\fp))_\red$ which contains $(s, \xi)^\flat$ and has dimension $\ge m - h^{0, 2} + 1 + h^{0, 1}$. The scheme-theoretic image of $C'$ on $\sS_{k(\fp)}$ hence has dimension $\ge m - h^{0, 2} + 1$. Then any irreducible component $C$ of $(\shP_{\sX/\sS})_\red$ which contains $C'$ will not satisfy the condition in (b). 

Finally, we prove (b). Note that $Z$ is reduced and irreducible, so that it has a generic point $\eta_{Z}$. Choose an LB-pair $(t, \zeta)$ such that $t$ lies above $\eta_Z$. We first argue that $\mathrm{char\,} \eta_{Z} = 0$. Indeed, otherwise we must have that the pair $(t, \zeta)$ does not lift to characteristic $0$, which contradicts \ref{mainthm1} that we just proved. Now the fact that $\mathrm{char\,} \eta_{Z} = 0$ automatically implies that $Z$ is flat over $\sO_{L}[1/N]$ by \ref{lem: Zariski closure is flat}, as $Z$ is the scheme-theoretic image of the morphism $\eta_{Z} \to \sS$. Moreover, we argue that $Z$ has codimension exactly $h^{0, 2}$ in $\sS$. We reduce to showing that $Z_{k(t)}$ has codimension exactly $h^{0, 2}$ in $\sS_{k(t)}$. This follows from \ref{multicharcodimlem} because $Z_{k(t)}$ is contained in the Zariksi closure of $\Def(\zeta)_{k(t)}$ on $\sS_{k(t)}$. Indeed, we may replace $\zeta$ by a power and assume that $\zeta = \zeta_0 + \zeta'$ for some $\zeta_0 \in \NS(\sX_\eta)_t$ and $\zeta' \in \NS(\sX_\eta)_t^\perp$. Then $\Def(\zeta)_{k(t)} = \Def(\zeta')_{k(t)}$. Note that $\mathrm{char\,} t = 0$, so that $\zeta'$ is automatically power-primitive. \qed


\begin{remark}
    Assuming that $\sO_L[1/N]$ is unramified over $\IZ[1/N]$ in \ref{sec: actual proof of main thms} is not strictly necessary for the argument if one is willing to keep track of possibly ramified extensions of $W(k(s))$. However, as we already have to increase $N$ in many other steps, it is not clear that we gain any generality in doing so. 
\end{remark}

\section{The $h^{0, 2} = 2$ Case}
\subsection{Theorem of the Fixed Part in a $p$-adic Setting}

\subsubsection{}\label{sec: setup 2} We shall define a variant of the general setup in \S\ref{sec: general setup}. The main difference is that we drop the assumption (i), and we will impose a two-step filtration on various cohomology sheaves. 


Let $L \subseteq \IC$ be a number field, $N \in \IN_{> 1}$ and $\sT$ is a connected smooth $\sO_L[1/N]$-scheme of finite type. Let $f : \sX \to \sT$ be an admissible family. Suppose that we are given a polarization $\bxi_0$ for this family. Let $\sH_\dR$ be the filtered flat vector bundle $R^2 f_* \Omega^\bullet_{\sX/\sT}$ on $\sT$ and let $\sH_\ell$ be the $\IQ_\ell$-local system $R^2 (f|_{\sT[1/\ell]})_{\et *} \IQ_\ell$ over $\sT[1/\ell]$. Let $\IV = (\IV_B, \IV_\dR)$ be the VHS given by on $\IV_B := R^2 f_{\IC *} \IQ$ and $\IV_\dR := \sH_\dR|_{\sT_\IC}$. As in \S\ref{sec: general setup}, $\bxi_0$ defines a symmetric bilinear pairing on $\IV$, $\sH_\dR$ and $\sH_\ell$. Let $\eta$ be the generic point of $\sT$ and $\bar{\eta}$ be the geometric point defined by a chosen algebraic closure of $k(\eta)$. Assume the following conditions: 

\begin{enumerate}[label=\upshape{(\roman*)}]
    \item $\sO_L$ is unramified at every finite prime $\fp \in \Spec(\sO_{L}[1/N])$.
    \item The pairing induced on $\sH_\dR$ is self-dual.
    \item $\mathrm{Mon}(\IV_B, s)$ is connected for some (and hence every) $s \in \sT(\IC)$. 
    \item For some (and hence every, by \cite[Prop.~6.14]{Larsen-Pink}) prime $\ell$, $\Mon(\sH_\ell|_{\sT_L}, \bar{\eta})$ is connected. 
    \item There exists a section $b : \Spec(\sO_L[1/N]) \to \sT$.
    \item Every prime dividing the order of $\NS(\sX_{\bar{\eta}})_{\mathrm{tor}}$ is a factor of $N$. 
\end{enumerate}
Note that these conditions can always be achieved up to enlarging $N$ and $L$, and replacing $\sT$ by a connected \'etale cover. For each closed point $\fp \in \Spec(\sO_L[1/N])$, we denote by $f_\fp$ the mod $\fp$ reduction of $f$ and $\hat{\sO}_{L, \fp}$ the completion of $\sO_L$ at $\fp$. Let $\sT_\fp := \sT \tensor k(\fp)$ and $\hat{\sT}_\fp := \sT \times_{\sO_{L}[1/N]} \mathrm{Spf}(\hat{\sO}_{L, \fp})$. We identify $\hat{\sO}_{L, \fp}$ with the ring of Witt vectors $W(k(\fp))$ of the residue field $k(\fp)$.

Let $\eta_\IC$ be the generic point of $\sT_\IC$. Then the relative Chern classes of the elements of $\NS(\sX_{\eta_\IC}) \simeq \Pic(\sX_{\IC})/\Pic(\sT_\IC)$ gives rise to a sub-VHS $\IV^0 = (\IV^0_B, \IV^0_\dR)$ of $\IV$. We remark that the assumption \ref{sec: setup 2}(iv) implies that every element of $\NS(\sX_{\bar{\eta}})_\IQ$ descends to $\eta$. Moreover, there is a natural identification of $\NS(\sX_{\bar{\eta}})_\IQ$ with $\H^0(\IV^0)$. Let $\IV^f_B$ be the maximal sub-local system of $\IV_B$ such that for some (and hence every) $s \in \sT(\IC)$, $\pi_1(\sT_\IC, s)$ acts trivially on $\IV^f_{B, s}$. By the theorem of the fixed part, there exists a filtered flat sub-bundle $\IV^f_\dR \subseteq \IV_\dR$ such that $(\IV^f_B, \IV^f_\dR)$ is a sub-VHS. Moreover, we have $\IV^0 \subseteq \IV^f$. We denote the orthogonal complement of $\IV^0$ in $\IV^f$ by $\IV^1$, and that of $\IV^f$ in $\IV$ by $\IV^2$. Note that if $f_\IC$ has large monodromy, then $\IV^1 = 0$ and $\IV^0 = \IV^f$, but we will be primarily interested in the case when this condition fails.


\begin{proposition}
\label{prop: spreading out splitting}
Up to further enlarging $N$, the following conditions hold. 
    \begin{enumerate}[label=\upshape{(\alph*)}]
        \item There exists a decomposition of $\ell$-adic local systems $\sH_\ell = \oplus_{i = 0}^2 \sH^i_\ell$ over $\sT[1/\ell]$ whose restriction to $\sT_\IC$ agrees with $\IV_B \tensor \IQ_\ell = \oplus_{i = 0}^2 \IV^i_B \tensor \IQ_\ell$. 
        \item There exists a decomposition $\sH_\dR = \oplus_{i = 0}^2 \sH^i_\dR$ of filtered flat vector bundles whose restriction to $\sT_\IC$ agrees with $\IV_\dR = \oplus_{i = 0}^2 \IV^i_\dR$. 
        \item For every finite prime $\fp \subseteq \sO_L[1/N]$ and $\sH_{\fp, \cris} := R^2 f_{\fp, \cris*} \sO_{\sX_\fp/ \hat{\sO}_{L, \fp}}$, there exists a decomposition of F-crystals $\sH_{\fp, \cris} = \oplus_{i = 0}^2 \sH^i_{\fp, \cris}$ such that under the crystalline-de Rham comparison isomorphism
        \begin{equation}
            \label{eqn: crystalline de Rham comparison}
            \sH_{\dR}|_{\hat{\sT}_\fp} \iso \sH_{\fp, \cris}(\hat{\sT}_\fp)\footnote{The notation $\sH_{\fp, \cris}(\hat{\sT}_\fp)$ means the coherent sheaf obtained by evaluating the crystal $\sH_{\fp, \cris}$ on $\hat{\sT}_\fp$, which is viewed as a pro-object in the crystalline site $\Cris(\sT_\fp / \hat{\sO}_{L, \fp})$.}
        \end{equation}
        the component $\sH^i_\dR|_{\hat{\sT}_\fp}$ is sent isomorphically onto $\sH^i_{\fp, \cris}(\hat{\sT}_\fp)$. Moreover, the pair $(\sH^i_\cris, \sH^i_\dR|_{\hat{\sT}_\fp})$ is a constant filtered F-crystal for $i = 0, 1$. 
        \item Let $s$ be an alg-geometric point over $\sT_\fp$ for some finite prime $\fp \subseteq \sO_L[1/N]$, and $\xi \in \Pic(\sX_s)$. Then $\shP(\sT, s, \xi)$ contains a point over the generic point of $\sT$ (resp. $\sT_\fp$) if and only if $c_{1, \cris}(\xi) \in \sH^0_{\fp, \cris}|_s$ (resp. $c_{1, \cris}(\xi) \in (\sH^0_{\fp, \cris} \oplus \sH^1_{\fp, \cris})|_s$).
    \end{enumerate}
\end{proposition}

Recall that $\shP(\sT, s, \xi)$ denotes the localization of $\shP_{\sX/\sT}$ at $(s, \xi)^\flat$ (cf. \ref{def: localizations}). 

\begin{proof}
 By Hironaka's resolution of singularities in characteristic $0$, we can find  a smooth projective $L$-variety $\overline{\sX}_L$ such that $\sX_L$ embeds as an open subvariety into $\overline{\sX}_L$. By enlarging $N$, we may assume that $\overline{\sX}_L$ extends to a smooth projective scheme $\overline{\sX}$ over $\sO_L[1/N]$ and $\sX$ is open in $\overline{\sX}$. Recall that in \ref{sec: setup 2}(v) we chose a section $b : \sO_L[1/N] \to \sT$. \mpr{Let $\iota: \sX_b \into \overline{\sX}$ denote the inclusion.} By the theorem of the fixed part, the natural map induced by $\iota_\IC := \iota \tensor \IC$ 
    \begin{equation}
        \label{eqn: invariant cycle}
        \iota_{\IC, B} : \H^2(\overline{\sX}_\IC, \IQ) \to \H^2(\sX_{b_\IC}, \IQ)
    \end{equation}
    is a surjection onto the $\pi_1(\sT_\IC, b)$-invariant part, i.e., onto $\IV^0_{B, b_\IC} \oplus \IV^1_{B, b_\IC}$. 

    (a) It suffices to show that $\IV^i_{b_\IC} \tensor \IQ_\ell$ is $\pi_1^\et(\sT[1/\ell], b_\IC)$-stable for $i = 0, 1$. We further reduce to showing that they are $\pi_1^\et(\sT_L, b_\IC)$-stable because the natural map $\pi_1^\et(\sT_L, b_\IC) \to \pi_1^\et(\sT[1/\ell], b_\IC)$ is surjective. Since $\NS(\sX_\eta)_\IQ = \NS(\sX_{\bar{\eta}})_\IQ = \H^0(\IV^0)$, the statement is already clear for $i = 0$. Therefore, it suffices to show the $\pi_1^\et(\sT_L, b_\IC)$-stablity of the direct sum $(\IV^0_{B, b_\IC} \oplus \IV^1_{B, b_\IC}) \tensor \IQ_\ell = \IV^f_{B, b_\IC} \tensor \IQ_\ell \subseteq \H^2_\et(\sX_{b_\IC}, \IQ_\ell)$. Since $b_L$ determines a splitting $\pi_1^\et(\sT_L, b_\IC) \iso \pi_1^\et(\sT_{\overline{L}}, b_\IC) \rtimes \Gal_{L}$ and $\pi_1^\et(\sT_{\overline{L}}, b_\IC) = \pi_1^\et(\sT_\IC, b_\IC)$ acts trivially on this direct sum, we reduce to showing that this direct sum is $\Gal_L$-stable. This follows easily from the fact that $\iota_{\IC, B} \tensor \IQ_\ell$ is $\Gal_L$-equivariant, as $\iota_\IC$ is defined over $L$. 

    (b) Set $\sH_{\dR, \IC} := \sH_\dR |_{\sT_\IC}$ and $\sH_{\dR, L} := \sH_\dR |_{\sT_L}$. Note that the horizontal subspace $\H^0(\sH_{\dR, \IC})^{\nabla = 0}$ is invariant under the $\Aut(\IC / L)$-action on $\H^0(\sH_{\dR, \IC})$, so that $\H^0(\sH_{\dR, L})^{\nabla = 0} \tensor_L \IC = \H^0(\sH_{\dR, \IC})^{\nabla = 0}$. Consider now the de Rham realization of $\iota_L := \iota \tensor L$: 
    $$ \iota_{L, \dR} : \H^2_\dR(\overline{\sX}_L/L) \to \H^2_\dR(\sX_{b_L}/L) $$
    This is a surjection onto the horizontal subspace $\H^0(\sH_{\dR, L})^{\nabla = 0} \subseteq \H^2_\dR(\sX_{b_L}/L)$---one can check this by tensoring with $\IC$ and comparing with (\ref{eqn: invariant cycle}). Recall that the theorem of the fixed part in particular says that there is a canonical Hodge filtration on $\H^0(\sH_{\dR, \IC})^{\nabla = 0}$. We observe that this Hodge filtration descends (necessarily uniquely) to $L$, as we can simply take the filtration on the image of $\iota_{L, \dR}$ inherited from that on $\H^2_\dR(\overline{\sX}_L/L)$. 
    
    Up to enlarging $N$, we may assume: 
    \begin{itemize}
        \item The direct summand $\NS(\sX_\eta) \tensor_\IZ \sO_{\sT_L} \subseteq \sH_{\dR, L}$ extends to a direct summand $\sH^0_\dR$ of $\sH_\dR$ such that the natural map $\H^0(\sH^0_\dR)^{\nabla = 0} \tensor \sO_{\sT} \to \sH^0_\dR$ is an isomorphism. 
        \item The direct summand $\H^0(\sH_{\dR, L})^{\nabla = 0} \tensor \sO_{\sT_L} \subseteq \sH_{\dR, L}$ extends to an orthogonal direct summand (as a filtered flat vector bundle) $\sH_\dR$, which we denote by $\sH^f_\dR$. 
        \item $\sH^0_\dR$ is an orthogonal direct summand  (as a filtered flat vector bundle) of $\sH^f_\dR$. 
    \end{itemize}
    Finally, we may take $\sH_\dR^1$ to be the orthogonal complement of $\sH^0_\dR$ in $\sH^f_\dR$. 


    (c) By the equivalence \cite[Cor.~6.8]{NO}, each piece $\sH^i_\dR|_{\hat{\sT}_\fp}$, as a vector bundle with an integrable quasi-nilpotent connection, gives rise to a direct summand $\sH^i_{\fp, \cris}$. We need to check that the decomposition $\sH_{\fp, \cris} = \oplus_{i = 0}^2 \sH^i_{\fp, \cris}$ respects the Frobenius structure, i.e., the canonical Frobenius action $F : F_{\sT_\fp/ \hat{\sO}_{L,  \fp}}^* \sH_{\fp, \cris} \to \sH_{\fp, \cris}$ sends each $F_{\sT_\fp/ \hat{\sO}_{L,  \fp}}^* \sH^i_{\fp, \cris}$ into $\sH^i_{\fp, \cris}$. Since this statement can be checked after inverting $p$, by \cite[Thm~4.1]{OgusDuke} it suffices to show the statement at one closed point on $\sT_\fp$. Therefore, it suffices to note that the Frobenius action on $\H^2_\cris(\sX_{b_\fp} / \hat{\sO}_{L, \fp}) = \sH_{\fp, \cris}|_{b_\fp}$ does send each $\sH^i_{\fp, \cris}|_{b_\fp}$ into $\sH^i_{\fp, \cris}|_{b_\fp}$. This is clear for $i = 0$; moreover, using the crystalline-de Rham comparison isomorphisms one checks that $(\sH^0_{\fp, \cris} \oplus \sH^1_{\fp, \cris})|_{b_\fp} \subseteq \H^2_\cris(\sX_{b_\fp} / \hat{\sO}_{L, \fp})$ is the image of $\H^2_\cris(\overline{\sX}_\fp / \hat{\sO}_{L, \fp}) \to \H^2_\cris(\sX_{b_\fp}/\hat{\sO}_{L, \fp})$ induced by $\iota \tensor k(\fp)$, and hence is stable under the Frobenius action. 

    Next, we check that $(\sH^i_{\fp, \cris}, \sH^i_\dR|_{\hat{\sT}_\fp})$ is a constant filtered F-crystal for $i = 0, 1$. This is clear for $i = 0$. Since $\sH^1_{\fp, \cris}$ is an F-crystal on $\Cris(\sT_\fp/\hat{\sO}_{L, \fp})$, its module of global sections $\H^0(\sH^1_\cris)$ is an F-crystal on $\Cris(k(\fp)/\hat{\sO}_{L, \fp})$, so that there is a natural morphism 
    \begin{equation}
    \label{eqn: cris trivialization}
        \H^0((\sT_\fp/\hat{\sO}_{L, \fp})_\cris, \sH^1_{\fp, \cris}) \tensor \sO_{\sT_\fp / \hat{\sO}_{L, \fp}} \to \sH^1_{\fp, \cris}
    \end{equation} of F-crystals, where $\sO_{\sT_\fp / \hat{\sO}_{L, \fp}}$ denotes the structure sheaf on $\Cris(\sT_\fp / \hat{\sO}_{L, \fp})$. Evaluating at the (pro)-object $\hat{\sT}_\fp \in \Cris(\sT_\fp/ \hat{\sO}_{L, \fp})$, we obtain
    \begin{equation}
        \label{eqn: dR trivialization}
        \H^0(\sH^1_{\dR})^{\nabla = 0} \tensor \sO_{\hat{\sT}_\fp} \to \sH^1_\dR |_{\hat{\sT}_\fp}
    \end{equation}
    which is an isomorphism of flat vector bundles. Therefore, (\ref{eqn: cris trivialization}) is also an isomorphism. Since (\ref{eqn: dR trivialization}) also preserves the filtrations, (\ref{eqn: cris trivialization}) and (\ref{eqn: dR trivialization}) together give a trivialization of the filtered F-crystal when $i = 1$.


    (d) Assume first that $\shP(\sT, s, \xi)$ dominates $\sT$. By \ref{prop: liftability}, we can find a lifting $(\wt{s},  \wt{\xi})$ of $(s, \xi)$ over $\sO_K$, where $K$ is a finite extension of $W(k(s))[1/p]$. Choose an embedding $K \into \IC$. Then $\shP(\sT, \wt{s}_\IC, \wt{\xi}_\IC)$ dominates $\sT_\IC$. Using \ref{lem: generic LB}(d), we infer that $c_{1, \dR}(\wt{\xi}_\IC)$ extends to an element of $\H^0(\IV^0)$. This implies that $c_{1, \dR}(\wt{\xi}) \in \sH^0_{\dR, \wt{s}} \subseteq \sH_{\dR, \wt{s}}$, and hence $c_{1, \cris}(\xi) \in \sH^0_{\fp, \cris}|_{s}$. Conversely, suppose that $c_{1, \cris}(\xi) \in \sH^0_{\fp, \cris}|_{s}$. To show that $\shP(\sT, s, \xi)$ dominates $\sT$ it suffices to show that $\xi$ deforms to the formal completion of $\sT \tensor W(k(s))$ at $s$. This follows from \cite[Prop.~1.12]{OgusCrystals} and the fact that $\sH^0_{\fp, \cris}$ is a trivial F-crystal and $\sH^0_{\dR} = \Fil^1 \sH^0_{\dR}$.

    The argument for the statement about $\sT_\fp$ is entirely similar. Assume now that $\shP(\sT_\fp, s, \xi)$ dominates $\sT_\fp$. Then using \ref{lem: generic LB}(d) again we infer that $c_{1, \cris}(\xi)$ extends as a global section of the isocrystal $\sH_{\fp, \cris}[1/p]$. Therefore, by \cite[Thm~1.4]{Morrow}, for some $m \in \IN$, $\xi^{p^m}$ extends to a line bundle on $\sX_\fp$ and hence on $\overline{\sX}_\fp$. This implies that $c_{1, \cris}(\xi)$ up to a multiple lies in the image of $\H^2_\cris(\overline{\sX}_\fp \tensor k(s) / W(k(s))) \to \H^2_\cris(\sX_s / W(k(s)))$. One easily checks that this image is precisely $(\sH^0_{\fp, \cris} \oplus \sH^1_{\fp, \cris})|_s$, which is a direct summand. Therefore, $c_{1, \cris}(\xi)$ must lie in $(\sH^0_{\fp, \cris} \oplus \sH^1_{\fp, \cris})|_s$ without passing to a multiple. Conversely, assume that $c_{1, \cris}(\xi) \in (\sH^0_{\fp, \cris} \oplus \sH^1_{\fp, \cris})|_s$. Then by \cite[Cor.~1.13]{OgusCrystals} and the fact that $\sH^0_{\fp, \cris} \oplus \sH^1_{\fp, \cris}$ is a constant filtered F-crystal, $\xi$ deforms to the formal completion of $\sT_\fp \tensor k(s)$ at $s$. Hence $\shP(\sT_\fp, s, \xi)$ dominates $\sT_\fp$. 
\end{proof}

\subsection{Non-liftable Line Bundles}
\label{sec: non-liftable}
In this section, we study the $h^{0, 2} = 2$ case in detail and prove \ref{thm: h = 2 case rough}. Again, apply our set-up and notations in \S\ref{sec: general setup} to the case when $R = \sO_L[1/N]$ for some number field $L \subseteq \IC$ and $N \in \IN$. Assume in addition that $h^{0, 2} = 2$ for the fibers of the family and for every (not just a general) geometric point $s \to \sS$, the Kodaira-Spencer map $$T_s \sS_{k(s)} \to \Hom(\H^1(\Omega_{\sX_s}), \H^2(\sO_{\sX_s}))$$ has rank at least $h^{0, 2}(\sX_s) = 2$. Let $\sE$ be the proper subscheme of $\sS$ obtained by applying \ref{multicharcodimlem}. Up to increasing $N$, assume that every irreducible components of $\sE$ have non-empty generic fiber over $\sO_L[1/N]$ and denote them by $\sE_i$'s. Note that the same $\sE$ also appeared in \ref{sec: actual proof of main thms}. 



\begin{remark}
    \label{rmk: KS on V^2} Our assumption on the Kodaira-Spencer map means that $\sE = E$ in \ref{multicharcodimlem} and \ref{Econstrprop} up to increasing $N$. In particular, we have 
    \[ \dim \sE_{i,\IC} \geq d = \dim_{R} \sS - (h^{0,2} - 1) = \dim_{R} \sS - 1 \]
    by \ref{Econstrprop}(ii) for all $i$. But $\dim \sE_{i,\IC} < \dim_{R} \sS$, so all the components $\sE_{i,\IC}$ have codimension $1$ in $\sS_{\IC}$. Note also that the restriction of $\IV$ to (the smooth locus of) $\sE_{i, \IC}$ is non-isotrivial.
\end{remark}

\subsubsection{} \label{sec: setup 2+} Consider in addition the following conditions. 
\begin{enumerate}[label=\upshape{(\Roman*)}]
    \item Each $\sE_i$ has geometrically connected fibers over $\sO_L[1/N]$. 
    \item There exists a scheme $\wt{\sE}_i$ smooth over $\sO_L[1/N]$ with geometrically connected fibers such that $\wt{\sE}_i$ admits a surjective and generically \'etale morphism to $\sE_i$. 
    \item When the base $\sT$ considered in \ref{sec: setup 2} is taken to be $\wt{\sE}_i$, the restricted family $(\sX/\sS)|_{\wt{\sE}_i}$ satisfies the assumptions in \ref{sec: setup 2} and the conclusion of \ref{prop: spreading out splitting} holds.
\end{enumerate}

We claim that the above conditions can always be satisfied up to replacing $L$ by a finite extension in $\IC$ and increasing $N$. First, up to extending $L$, we may assume that each $\sE_{i, L}$ is geometrically connected, so that (I) holds up to increasing $N$ by \cite[055H]{stacks-project}. Next, to construct $\wt{\sE}_i$ which satisfies (II) and (III), we first apply Hironaka's resolution of singularities to find a smooth $\sE'_{i, L}$ admiting a proper birational map to $\sE_{i, L}$, then we find a connected \'etale cover $\wt{\sE}_{i, \IC}$ of $\sE'_{i, \IC} := \sE'_{i, L} \tensor \IC$ so that \ref{sec: setup 2}(iii) is satisfied, i.e., $\Mon(\IV_B|_{\wt{\sE}_{i, \IC}}, s)$ is connected for some (and hence every) $s \in \wt{\sE}_{i, \IC}$. Then we apply the argument of \ref{lem: replacement}: Up to replacing $L$ by a finite extension in $\IC$, $\wt{\sE}_{i, \IC}$ is defined over $L$, i.e., $\wt{\sE}_{i, \IC} = \wt{\sE}_{i, L} \tensor \IC$ for some \'etale cover $\wt{\sE}_{i, L}$ of $\sE_{i, L}$; moreover, we may assume that $\wt{\sE}_{i, L}(L) \neq \emptyset$. Up to replacing $\wt{\sE}_{i, L}$ (and hence also $\wt{\sE}_{i, \IC}$) by a connected \'etale cover, \ref{sec: setup 2}(iv) is satisfied, i.e., $\Mon(\sH_\ell|_{\wt{\sE}_{i, L}}, s)$ is connected for some (and hence every) geometric point $s$ on $\wt{\sE}_{i, L}$ for every prime $\ell$. Up to increasing $N$ and applying \cite[055H]{stacks-project} again, $\wt{\sE}_{i, L} \to \sE_{i, L}$ spreads out to a morphism $\wt{\sE}_i \to \wt{\sE}$ which satisfies (II) above. Finally, the rest of assumptions in \ref{sec: setup 2} when $\sT$ there is taken to be $\wt{\sE}_i$ can easily be satisfied up to increasing $N$ again, so that (III) is met.

\begin{proposition}
\label{prop: only non-liftable situation} Assume that the conditions in \ref{sec: setup 2+} are satisfied. Let $\fp \subseteq \sO_L[1/N]$ be a finite prime. We denote by $\sH_{\fp, \cris}|_{\wt{\sE}_i} = \oplus_{j = 0}^2 \sH^j_{i, \fp, \cris}$ the decomposition provided by \ref{prop: spreading out splitting}(c) when $\sT$ is taken to be $\wt{\sE}_i$. Then following are equivalent for a LB-pair $(s, \xi)$ on $\sX/\sS$ when $s$ lies above $\sS_\fp$. 
    \begin{enumerate}[label=\upshape{(\roman*)}]
        \item There exists an $\sE_i$ and a lift $u \to \wt{\sE}_i$ lifting $s$ such that, identifying $\sX_s$ with $\sX_u$, $c_{1, \cris}(\xi) \in \H^2_\cris(\sX_u/W(k(s)))$ lies in $(\sH^0_{i, \fp, \cris} \oplus \sH^1_{i, \fp, \cris})|_u$ but not in $\sH^0_{i, \fp, \cris}|_u$. 
        \item The pair $(s, \xi)$ is not strongly liftable in the sense of \ref{def: localizations}. In particular, $\shP(\sS, s, \xi)$ is not flat. 
    \end{enumerate}
\end{proposition}

\begin{proof} Set $k:= k(s)$, $p := \mathrm{char\,}k$, and $W := W(k)$.

    (i) $\Rightarrow$ (ii): Let $\eta_{i, \fp}$ be the generic point of $\sE_{i, \fp}$. We first claim that the image of $\shP(\sS, s, \xi)$ on $\sS$ contains $\eta_{i, \fp}$. Let $\Delta$ be the formal completion of $\wt{\sE}_{i, \fp} \tensor_{k(\fp)} k$ at $u$. Then $\Delta = \Spf(\sO_\Delta)$ for a formal power series ring $\sO_\Delta$ over $k$. As $(\sH^0_{i, \fp, \cris} \oplus \sH^1_{i, \fp, \cris})$ is a constant filtered F-crystal over $\wt{\sE}_{i, \fp}$ by \ref{prop: spreading out splitting}(c), \cite[Cor.~1.13]{OgusCrystals} implies that $\xi$ extends to $\Delta$. Therefore, we can find an LB-pair $(t, \zeta)$ such that as an alg-geometric point on $\shP_{\sX/\sS}$, $(t, \zeta)$ specializes to $(s, \xi)$, and $t \to \sS$ factors through $\eta_{i, \fp}$. 
    
    Next, we argue that $(t, \zeta)^\flat \in \shP(\sS, s, \xi)$ does not lie in the Zariski closure of any characteristic $0$ point of $\shP(\sS, s, \xi)$. Suppose on the contrary that there exists an LB-pair $(\wt{t}, \wt{\zeta})$ in characteristic $0$ which specializes to $(t, \zeta)$. Since $t^\flat \in \sS$ is just $\eta_{i, \fp}$, it is easy to see that the Zariski closure of $\wt{t}^\flat$ in $\sS$ is of codimension $1$ by \cite[0B2J]{stacks-project}. Note that $\wt{t}^\flat \in \sS$ cannot be the generic point of $\sE_{i}$. Indeed, otherwise $\shP(\wt{\sE}_{i}, u, \xi)$ dominates $\wt{\sE}_i$, which by \ref{prop: spreading out splitting}(d) implies that $\xi \in \sH^0_{i, \fp, \cris}|_u$ and hence contradicts (i). By considering dimensions of Zariski closures, it is clear that $\wt{t}^\flat \in \sS$ cannot be the generic point of any $\sE_j$ for $j \neq i$, so we must have that $\wt{t}^\flat \not\in \sE$. However, this implies that the deformation space $\Def(\zeta)$ of $\zeta$ in the formal completion of $\sS_{k(\wt{t})}$ at $t$ has codimension $1$ and $\wt{t} \not\in \sE_{k(\wt{t})}$. This contradicts \ref{multicharcodimlem}.

    (ii) $\Rightarrow$ (i): Consider the deformation space $\Def(\xi)_k$ (resp. $\Def(\xi)$) in the formal completion $\hat{\sS}_{s, k}$ (resp. $\hat{\sS}_s$) of $\sS \tensor k$ (resp. $\sS \tensor W$) at $s$. Recall that if $\Def(\xi)_k \subseteq \hat{\sS}_{s, k}$ has codimension exactly $2$, then $\Def(\xi)$ is flat over $W$ by miracle flatness \cite[00R4]{stacks-project}, which will imply that $(s, \xi)$ is strongly liftable. Therefore, we must have that $\Def(\xi)_k$ has codimension $1$ in $\hat{\sS}_{s, k}$. In turn, this implies that there exists an LB-pair $(t, \zeta)$ which specializes to $(s, \xi)$ such that the Zariski closure $Z$ of $t^\flat$ in $\sS$ lies in $\sS_\fp$ and has codimension $1$ in $\sS_\fp$. If we form $\Def(\zeta)_{k(t)}$ inside $\hat{\sS}_{t, k(t)}$ analogously, then $\Def(\zeta)_{k(t)}$ must have codimension $1$. By \ref{multicharcodimlem} (see also \ref{rmk: primitive implies general}) we must have $t \in \sE_{k(t)}$. Therefore, the Zariski closure of $t^\flat$ in $\sS$ must be $\sE_{i, \fp}$ for some $\sE_i$. Let $u \in \wt{\sE}_{i, \fp}(k)$ be any lift of $s$. Then $\shP(\wt{\sE}_{i, \fp}, u, \xi)$ dominates $\wt{\sE}_{i, \fp}$, but $\shP(\wt{\sE}_i, u, \xi)$ does not dominate $\wt{\sE}_i$. Hence the conclusion follows from \ref{prop: spreading out splitting}(d). 
\end{proof}

Finally, we are ready to prove \ref{thm: h = 2 case rough}.
\begin{theorem}
\label{thm: h = 2 case}
Consider the setting in \ref{thm: h = 2 case rough}. Namely, let $L \subseteq \IC$ be a number field, $N \in \IN_{> 1}$, $\sS$ be a smooth quasi-projective $\sO_L[1/N]$-scheme with $\sS_\IC$ connected, and let $f : \sX \to \sS$ be a smooth projective family with geometrically connected fibers. Assume the following conditions. 
\begin{itemize}
    \item $f_\IC$ has large monodromy. 
    \item For every $s \in \sS(\IC)$, $h^{0, 2}(\sX_s) = 2$, and the Kodaira-Spencer map $T_s \sS_\IC \to \Hom(\H^1(\Omega_{\sX_s}), \H^2(\sO_{\sX_s}))$ has rank $\ge 2$. 
\end{itemize}
Then up to increasing $N$, the following holds: given any LB-pair $(s, \xi)$ on $\sX/\sS$ with $k := k(s)$, there exists another LB-pair $(s', \xi')$ over $k$ with the following properties. 
    \begin{enumerate}[label=\upshape{(\roman*)}]
        \item There exists a smooth connected $k$-curve $C$ equipped with a morphism $C \to \sS$, a relative line bundle $\bxi \in \shP_{\sX_C/C}(C)$, and geometric points $c, c'$ on $C$ lifting $s, s'$, such that identifying $\sX_{c}$ with $\sX_{s}$ and $\sX_{c'}$ with $\sX_{s'}$ we have $\bxi_c \sim_\IQ \xi$ and $\bxi_{c'} \sim_\IQ \xi'$; 
        \item There exist $\xi'_m \in \Pic(\sX_{s'})$ for $m = 1, 2$ such that $\xi' = \xi'_1 + \xi'_2$ and each $(s', \xi_m')$ is liftable. 
    \end{enumerate}
\end{theorem}

\begin{proof}
    To prove the theorem, we are allowed to make replacement as in \ref{lem: replacement} so that the family $f : \sX \to \sS$ satisfies the assumptions in \S\ref{sec: general setup}. Moreover, up to replacing $L$ by a further finite extension in $\IC$ and increasing $N$, the family satisfies the assumptions in \ref{sec: setup 2+}, so that we are in the setting of \ref{prop: only non-liftable situation}. Let $\IV|_{\wt{\sE}_{i, \IC}} = \oplus_{j = 0}^2 \IV^j_i$ be the decomposition obtained by applying the set-up in \ref{sec: setup 2} when the base $\sT$ therein is taken to be $\wt{\sE}_i$. We make the following claim. 

    \vspace{1em}

    \noindent \textit{Claim: Up to increasing $N$, the following holds: For every $i$ such that $\IV^1_i \neq 0$ and for every finite prime $\fp \subseteq \sO_L[1/N]$, there exists a point $s'_0 \in \sE_{i, \fp}(\overline{k(\fp)})$ such that there exists a lift $u'_0 \in \wt{\sE}_{i, \fp}(\overline{k(\fp)})$ such that, identifying $\sX_{s'_0}$ with $\sX_{u'_0}$, there exists a line bundle $\zeta \in \Pic(\sX_{s'_0})$ for which $c_{1, \cris}(\zeta) \in \sH^2_{i, \fp, \cris}|_{u'_0}$.} 

    \vspace{0.5em}

    \noindent \textit{Proof of claim.} Note first that by \ref{rmk: KS on V^2} $\IV|_{\wt{\sE}_{i, \IC}}$, or equivalently, $\IV^2_{i}$ is a non-isotrivial $\IQ$-VHS. Moreover, we argue that $h^{2, 0}(\IV^2_i) = \mathrm{rank\,} \Fil^2 \IV^2_{i, \dR} = 1$, or equivalently, $h^{2, 0}(\IV^2_i) = 1$. Indeed, any global section of $\IV_B|_{\wt{\sE}_{i, \IC}}$ which is everywhere of type $(1, 1)$ must lie in $\IV^0_i$. Since $\IV^1_i$ is assumed non-zero, $h^{2,0}(\IV^1_i)$ must be nonzero, but since $\IV^2_2$ is not isotrivial, $h^{2,0}(\IV^1_i)$ cannot be $2$. Hence by \cite[Prop.~6.4]{Moonen} there exists some $\IC$-point $u'_\IC$ on $\wt{\sE}_{i}(\IC)$ such that $\IV^2_{i, u'_\IC}$ has a nonzero $(1, 1)$-class, which corresponds to a line bundle $\zeta_\IC \in \Pic(\sX_{u'_\IC})$. Up to increasing $N$, for any $\fp$ we may find some $u'_0 \in \wt{\sE}_{i, \fp}(\overline{k(\fp)})$ lying in the Zariski closure of $(u_\IC')^{\flat}$ in $\wt{\sE}_{i}$ and take $s'_0$ to be the image of $u'_0$. The desired line bundle $\zeta$ is constructed by specializing $\zeta_\IC$.  \qed \vspace{1em}

    Now suppose that we have already increased $N$ so that the claim holds and $(s, \xi)$ is a LB-pair as in the statement of the theorem. If suffices to treat the case when it is not strongly liftable. Let $\fp$ be the image of $s$ on $\Spec(\sO_L[1/N])$. By \ref{prop: only non-liftable situation} and its proof, there exists an $\sE_i$ and a lift $u \in \wt{\sE}_{i, \fp}(k)$ such that $c_{1, \cris}(\xi)$ lies in $(\sH^0_{i, \fp, \cris} \oplus \sH^1_{i, \fp, \cris})|_{u}$ but not in $\sH^0_{i, \fp, \cris}|_{u}$. Note that this in particular implies that $\IV^1_i \neq 0$. Moreover, $\shP(\wt{\sE}_{i, \fp}, u, \xi)$ dominates $\sE_{i, \fp}$. Recall that in the construction of $\wt{\sE}_{i, \fp}$, we assumed that $\mathrm{Mon}(\sH_{i, \ell}, u)$ is connected (this is subsumed into \ref{sec: setup 2+}(III)). By \ref{lem: generic LB}(d), there exists a relative line bundle $\wt{\bxi}$ on $\sX|_{\wt{\sE}_{i, \fp}}$ such that $\wt{\bxi}_u \sim_\IQ \xi$. Fix this $\sE_i$ and take $(s'_0, u'_0, \zeta)$ as above. Let $s'$ and $u'$ be the base change of $s'_0$ and $u'_0$ along some embedding $\overline{k(\fp)} \into k$ determined by $s$. Then $\wt{\bxi}_{u'}$ is a line bundle on $\sX_{u'} = \sX_{s'}$. As $c_{1, \cris}(\bxi_{u}) \in \sH^2_{i, \fp, \cris}|_{u}$, by \cite[Thm~4.1]{OgusDuke} the relative crystalline Chern class of $\bxi$ must be a global section of the direct summand $\sH^2_{i, \fp, \cris}$ in $\sH_{i, \fp, \cris}$, and hence $c_{1, \cris}(\bxi_{u'}) \in \sH^2_{i, \fp, \cris}|_{u'}$ as well. By \ref{prop: only non-liftable situation}, $\zeta$ and $\zeta - \wt{\bxi}_{u'}$ are both line bundles on $\sX_{s'}$ which are strongly liftable because their $c_{1, \cris}$ in $\sH_{i, \fp, \cris}|_{u'}$ has a non-zero component in $\sH^2_{i, \fp, \cris}|_{u'}$. 
    
    To conclude the proof, it suffices to note that by \cite[Cor.~1.9]{Bertini-irred}, we can always find an irreducible $k$-curve $C' \subseteq \wt{\sE}_{i, \fp} \tensor k$ which contains both $u$ and $u'$. Then we may simply take $C$ to be the normalization of $C'$, $\bxi$ to be the restriction of $\wt{\bxi}$ to $C$, and $c, c'$ to be any lifts of $u, u'$ respectively. 
\end{proof}

\section{Applications and Examples}
\label{sec: examples}
We introduce the following notations: Let $B$ be a base scheme and $\sV$ be a vector bundle of rank $r$ over $B$. We denote by $\IA(\sV)$ the associated $\IA^r$-bundle over $B$, $\IA(\sV)^*$ the open part $\IA(\sV)$ minus the zero section, and $\IP(\sV)$ the projectivization of $\sV$. 

\subsection{Elliptic Surfaces}
Let $\pi : X \to \IP^1_k$ be an elliptic surface (with zero section) over a field $k$. Define $L := (R^1 \pi_* \sO_X)^\vee$ to be the fundamental line bundle and set $h := \deg(L)$ to be the \textit{height} of $L$. Note that if $h = 1$, then $X$ is rational, and if $h = 2$, then $X$ is an elliptic K3 surface. In general, we have that the irregularity $q = 0$ and $h = \sX(\sO_X) = p_g - 1$, where $p_g = h^{2, 0}(X)$ is the geometric genus. We say that $X$ is \textit{admissible} if all fibers of $\pi$ are irreducible.


In order to parametrize all $X$ with height $h$, we consider $\IP^1$ over $\IZ[1/6]$ and let $\varpi$ be the structural morphism $\IP^1 \to \Spec(\IZ[1/6])$. Set $\sV_i$ to be the vector bundle over $\IZ[1/6]$ given by $\varpi_* \sO_{\IP^1}(ih)$. Let the coordinates of $\sP := \IA(\sV_4 \oplus \sV_6)^* \times_{\IZ[1/6]} \IP(\sO_{\IP^1}(2h) \oplus \sO_{\IP^1}(3h) \oplus \sO_{\IP^1})$ be denoted by $( (a_4, a_6), [x : y : z])$ and set $\overline{\sX}$ to be the subscheme defined by (cf. \cite[Thm~1]{Kas})
$$ y^2 z = x^3 - a_4 xz^2 - a_6 z^3. $$
Note that the morphism $\sP \to \IA(\sV_4 \oplus \sV_6)^*$ carries a natural action by $\Aut(\IP^1)$ and $\IG_m$, the latter is given by $$ \lambda \cdot ((a_4, a_6), [x : y : z]) \mapsto ((\lambda^4 a_4, \lambda^6 a_6), [\lambda^2 x : \lambda^3 y : z]).  $$
These actions commute and both stabilize $\sP$. Let $\sS \subseteq \IA(\sV_4 \oplus \sV_6)^*$ be the open subscheme over which $\overline{\sX}$ is smooth and $\sX$ be the restriction of $\overline{\sX}$ to $\sS$.  It is well known that a geometric point $s \to \IA(\sV_4 \oplus \sV_6)^*$ factors through $\sS$ if and only if the fiber $\sX_s$ is an admissible elliptic surface. 

The moduli stack $\sfM$ over $\IZ[1/6]$ of admissible elliptic surfaces over $\IP^1$ of height $h$ is given by the quotient of $\sS$ by the action of $\IG_m \times \Aut(\IP^1)$. One easily computes that $\sfM$ has relative dimension $10 h - 2$ over $\IZ[1/6]$. To avoid working with stacks, we consider $\sS$. 

\begin{lemma}
\label{lem: Torelli for elliptic}
    Assume that $h \ge 3$. Then for every $s \in \sS(\IC)$, the Kodaira-Spencer map $T_s \sS_\IC \to \Hom(\H^1(\Omega_{\sX_s}), \H^2(\sO_{\sX_s}))$ has rank $10 h - 2$. 
\end{lemma}
\begin{proof}
    This is a consequence of \cite[Thm~2.10]{SB}. We just need to explain why the hypothesis there is satisfied: Note that the Euler characteristic of a nodal curve (resp. cuspidal) of $1$ (resp. $2$), and that of $\sX_s$ is $12 \chi(\sO_X) = 12 h$. Therefore, the topological Hurwitz formula tells us that there are at least $6h$ singular fibers. The line bundle $L$ in \textit{loc. cit.} is the same as our $L$, i.e., the fundamental line bundle. 
\end{proof}

\begin{theorem}
There exists an open dense subscheme $U \subseteq \sS$ such that every LB-pair $(s, \xi)$ with $s$ lying on $U$ is liftable. 
\end{theorem} 
\begin{proof}
    This follows from \ref{mainthm1}, \ref{lem: Torelli for elliptic} and the monodromy calculation \cite[Thm~4.10]{JF}. 
\end{proof}

\begin{theorem}
    For $p \gg 0$, the following is true: Let $k$ be an algebraically closed field of characteristic $p$, $X$ be an admissible elliptic surface over $\IP^1$ of height $3$, and $\xi$ be any line bundle on $X$. Then some power of $\xi$ deforms to a sum of two liftable ones. 
    
    More precisely, there exist a smooth connected $k$-curve $C$ and a smooth family $\sX \to C \times \IP^1 \to C$ whose geometric fibers over $C$ are height $3$ elliptic surfaces over $\IP^1$ such that: 
    \begin{enumerate}[label=\upshape{(\roman*)}]
        \item there exist $c \in C(k)$ and $\bxi \in \shP_{\sX/C}(C)$, and an isomorphism $\sX_c \cong X$ through which $\bxi_c \sim_\IQ \xi$; 
        \item for some $c' \in C(k)$, $\bxi_{c'} \sim_\IQ \zeta_1 + \zeta_2$ such that $(\sX_{c'}, \zeta_i)$ for both $i = 1, 2$ are liftable. 
    \end{enumerate}
\end{theorem}
\begin{proof}
    Apply \ref{thm: h = 2 case} to the family $\sX \to \sS$ constructed above. 
\end{proof}


\subsection{Complete Intersection Surfaces}

Fix some $r \in \IN_{\ge 1}$ and a tuple $(d_1, \cdots, d_r) \in \IN_{\ge 1}^{\oplus r}$. Consider $\IP^{r + 2}$ over $\IZ$ and let $\varpi : \IP^{r + 2} \to \Spec(\IZ)$ denote the structural morphism. Let 
\[ \overline{\sX} \subseteq \IP^{r + 2} \times_\IZ \prod_{i = 1}^r \IP( \varpi_* \sO_{\IP^{r + 2}}(d_i)), \hspace{2em} \left(\textrm{resp. } \sS \subseteq \prod_{i = 1}^r \IP( \varpi_* \sO_{\IP^{r + 2}}(d_i)) \right) \]
be the universal complete intersection of $r$ hypersurfaces of degree $d_1, \cdots, d_r$ (resp. the open subscheme over which $\overline{\sX}$ is smooth). Denote the restriction of $\overline{\sX}$ to $\sS$ by $\sX$. 

\begin{theorem}
\label{thm: surfaces in P3}
    Assume the following two numerical conditions: 
    \begin{enumerate}[label=\upshape{(\roman*)}]
        \item $\sum d_i > r + 3$
        \item for some (and hence every) $s \in \sS(\IC)$, $\dim \H^1(T_{\sX_s}) \ge h^{2, 0}(\sX_s)$. 
    \end{enumerate}
    Then there exists an open dense $U \subseteq \sS$ such that every LB-pair $(s, \xi)$ with $s$ lying on $U$ is liftable. 
\end{theorem}
\begin{proof}
    For every $s \in \sS$, the canonical divisor $K_{\sX_s}$ of $\sX_s$ is the restriction of the $(- r - 3 + \sum d_i)$-multiple of the hyperplane section on $\IP^{r + 2}_{k(s)}$, so condition $(i)$ ensures that $K_{\sX_s}$ is ample. By \cite[Cor.~1.9, Thm~3.5]{Peters} and their proofs, one infers that the natural map $T_s \sS_\IC \to \H^1(T_{\sX_s})$ is surjective for every $s \in \sS(\IC)$. Then Thm~5.4 \textit{ibid} further tells us that the Kodaira-Spencer map $\H^1(T_{\sX_s}) \to \Hom(\H^1(\Omega_{\sX_s}), \H^2(\sO_{\sX_s}))$ is injective. Finally, in order to apply \ref{mainthm1} to $\sS$, it suffices to check the large monodromy condition on $\sS_\IC$. This follows from \cite[Thm~5]{Beauville}. 
\end{proof}

\begin{remark}
    The formulas for the numbers $h^{p, q}(\sX_s)$ and $\dim \H^1(T_{\sX_s})$ for some (and hence every) $s \in \sS(\IC)$, in terms of $d_1, \cdots, d_r$ are completely known \cite[théorème 2.3]{zbMATH03417601}, so (ii) is indeed a purely numerical condition. In particular, when $r = 1$, any $d_1 \ge 5$ will satisfy the hypothesis of \ref{thm: surfaces in P3}. This gives \ref{thm: surfaces in P3 intro}. 
\end{remark}

\subsection{Non-examples}
\label{sec: non-examples}

We remark on the assumptions (a) $f_\IC$ has large monodromy and (b) $f_\IC$ generically has sufficiently big period image in \ref{mainthm1} and \ref{mainthm2}. The upshot is that (a) should be relaxable whereas (b) is more rigid. 

To begin, there is a quick counterexample to the conclusion when neither (a) nor (b) holds: Suppose that $\sS = \Spec(\sO_L[1/N])$ and $\sX$ is a K3 surface over $\sS$. Then for infinitely many primes $\fp$, the reduction $\sX_{k(\fp)}$ has higher geometric Picard number than $\sX_L$ (\cite[Thm~1.1]{SSTT}, \cite[Thm~1.1]{Tayou}). At such $\fp$, we can find an irreducible component of $\shPic_{\sX/\sS}$ which admits a $\overline{k(\fp)}$-point but not an $\overline{L}$-point, and hence is not flat over $\sS$. Therefore, both \ref{mainthm1} and \ref{mainthm2} are false in this case. On the other hand, for families with $h^{0, 2} = 1$, we can relax (a) to ``$\mathrm{Mon}(R^2f_{\IC} \IQ, s)$ is positive dimensional for some (and hence every) $s \in \sS(\IC)$'', which is in fact equivalent to (b). This is a key observation we used (somewhat implicitly) in \cite{HYZ}.

The example below illustrates that (conjecturally) even if (a) is satisfied, (b) cannot be dropped: 

\begin{proposition}
Let $\sS$ be as in \ref{thm: surfaces in P3} with $r = 1$ and $d_1 = d$ being an even number $\ge 6$. Let $\sC \subseteq \sS$ be a closed subscheme which is smooth and of relative dimension $1$ over an open part of $\IZ$, and such that $\sC_{\IC}$ contains a Hodge-generic point for the VHS given by $\H^2$ of $\sX_\IC$ over $\sS_\IC$. 

Let $\sC_{\mathrm{bad}}$ be the set of closed points $s \in \sC$ such that there exist some alg-geometric point $\bar{s}$ over $s$ and $\xi \in \Pic(\sX_{\bar{s}})$ for which the pair $(\bar{s}, \xi)$ is non-liftable on $\sC$. If we assume the (divisorial) Tate conjecture over finite fields and the Zilber-Pink conjecture, then $\sC_{\mathrm{bad}}$ is Zariski dense on $\sC$. 
\end{proposition}

\begin{proof}
    Let $s \in \sS$ be any closed point and $\ell$ be a prime $\neq \mathrm{char\,} k(s)$. The Weil conjectures imply that the Frobenius eigenvalues on $\H^2_\et(\sX_{\bar{s}}, \IQ_\ell(1))$ which are not roots of unity come in pairs. As $\H^2_\et(\sX_{\bar{s}}, \IQ_\ell(1))$ is even dimensional by the assumption that $d$ is even, the Tate conjecture implies that $\mathrm{rank\,} \NS(\sX_{\bar{s}})$ is even. Therefore, there exists a line bundle $\xi \in \Pic(\sX_{\bar{s}})$ which does not have the same class as $\sO_{\sX_{\bar{s}}}(1)$. 

    Suppose now that the LB-pair $(\bar{s}, \xi)$ lifts to characteristic $0$. Then by choosing an isomorphism between $\IC$ and the algebraic closure of $W(k(\bar{s}))[1/p]$, we obtain a $\IC$-point on the NL locus. Because $\sC$ is Hodge-generic, the expected codimension of a NL locus in $\sC_{\IC}$ is $h^{0,2} > 1$. By the Zilber-Pink conjecture formulated in \cite[Conj. 2.5]{BKU} this implies the NL locus of $\sC$ is a finite set of points, which implies that $s$ lies in the Zariski closure of finitely many $\IC$-points on $\sC_\IC$. 
\end{proof}

\appendix

\section{A Lemma on Flat Frames}
\begin{lemma}
    Let $k$ be a field and fix some $r \in \IN$. Let $A_\infty = k[\![x_1, \cdots, x_n]\!]$ for some $n$. Set $\ID := \Spf(A_\infty)$, $\fm := (x_1, \cdots, x_n)$, $I := (x_1^{r + 1}, \cdots, x_i^{r + 1})$ and $A := A_\infty/I$. Let $(\hat{M}, \hat{\nabla})$ be any flat vector bundle over $\ID$ and $(M, \nabla)$ be its restriction to $\Spec(A)$. Suppose that $r!$ is invertible in $k$. 
    
    Then for every $m_0 \in M / \fm M$, there exists a unique $m \in M^{\nabla = 0}$ such that $m_0 = m \mod \fm M$. In particular, $(M, \nabla)$ is trivial as a flat vector bundle. 
\end{lemma} 

\begin{proof}
First, we note that $\hat{\nabla}_{\p_i} : \hat{M} \to \hat{M}$ has a standard meaning, where $\{\p_1, \cdots, \p_n\}$ are dual to the generators $dx_1, \cdots, dx_n$ for $\Omega_{\ID / k}$. Let us simply write $\hat{\nabla}_i$ for $\hat{\nabla}_{\p_i}$. The integrability of $\hat{\nabla}$ implies that the operators $\hat{\nabla}_i$'s commute. Now we make sense of $\nabla_i$. For each $i$, set $I^i := (x_1^{r + 1}, \cdots, x_i^{r}, \cdots, x_n^{r + 1})$ and $A^i := A / (x_i^{r}) = A_\infty / I^i$. By applying the conormal sequence to the surjection $k[x_1, \cdots, x_n] \to A$, one checks that 
$$\Omega_{A/ k} = \bigoplus_{i = 1}^n  A^{i} dx_i \simeq \bigoplus_{i = 1}^n  A^{i} $$
Then we define $\nabla_i$ to be the composition $M \stackrel{\nabla}{\to} M \tensor_A \Omega_{A/k} \to M \tensor_A A^{i}$, where the second map is the projection to the $i$th factor. As the restriction of $\hat{\nabla}$, $\nabla$ is characterized by the property that for any $\hat{m} \in \hat{M}$ with image $m \in M$, $\nabla(m)$ is the image of $\hat{\nabla}(\hat{m})$ under the composition $\hat{M} \tensor_{A_\infty} \Omega_{\ID/k} \to M \tensor_A (\Omega_{\ID/k} \tensor_{A_\infty} A) \to M \tensor_A \Omega_{A/k}$, so that the following diagram commutes: 
\begin{equation}
\label{eqn: nabla_i}
\begin{tikzcd}
	{\hat{M}} & {\hat{M}} \\
	M & M \tensor_A A^i
	\arrow["{\hat{\nabla}_{i}}", from=1-1, to=1-2]
	\arrow[from=1-1, to=2-1]
	\arrow["{\nabla_{i}}", from=2-1, to=2-2]
	\arrow[from=1-2, to=2-2]
\end{tikzcd}    
\end{equation}

Define an additive map $\hat{P} : \hat{M} \to \hat{M}$ by the formula (cf. \cite[(5.1.2)]{Katz}): 
\begin{equation}
    \label{eqn: Katz}
    \hat{P}(\hat{m}) = \sum_{\alpha} \frac{(-1)^{|\alpha|}}{\alpha!} x^\alpha \hat{\nabla}^\alpha \text{ for all } \hat{m} \in \hat{M}
\end{equation}
where $\alpha$ runs through tuples $(\alpha_1, \cdots, \alpha_n) \in (\IN_{\le r})^{n}$, $|\alpha| = \sum_i \alpha_i$, $\alpha! = \alpha_1! \cdots \alpha_n!$, $x^\alpha = x_1^{\alpha_1} \cdots x_n^{\alpha_n}$ and $\hat{\nabla}^\alpha = \hat{\nabla}^{\alpha_1}_1 \cdots \hat{\nabla}^{\alpha_n}_n$. We define an additive map $P : M \to M$ by the setting for each $m \in M$ and $\hat{m} \in \hat{M}$ lifting $m$, $P(m) = \hat{P}(\hat{m}) \mod IM$. This is well defined because for each $i$, $\hat{P}(x_i^{r + 1} \hat{m})$ is divisible by $x_i^{r + 1} \in I$. We claim that $P$ is a projection onto $M^{\nabla = 0}$. By the diagram (\ref{eqn: nabla_i}), we reduce to showing the following two statements for $\hat{m} \in \hat{M}$: (a) For every $i$, $\hat{\nabla}_i P(\hat{m}) \in I^i \hat{M}$. (b) If for every $i$, $\hat{\nabla}_i (\hat{m}) \in I^i \hat{M}$, then $\hat{P}(\hat{m}) = \hat{m} \mod IM$. Both are readily checked by a direct computation. 

Finally, for any $m_0 \in M/ \fm M$, to construct $m$ as desired we may start with any $m \in M$ with mod $\fm M$ class $m_0$ and replace $m$ by $P(m)$---it is clear by (\ref{eqn: Katz}) this does not change mod $\fm$ class. Now we already know that $(M, \nabla)$ is isomorphic to $\mathrm{rank\,} M$ copies of $(A, d)$. To verify the uniqueness of $m$ if suffices to note that if $d f = 0$ and $f(0) = 0$ for $f \in A$, then $f = 0$. This is clear when $\mathrm{char\,} k = 0$, and one uses the assumption $r < p$ when $\mathrm{char\,} k =p > 0$. 
\end{proof}

\begin{remark}
    When one relates the above to \ref{rlimpdef}, note that $(x_1^{r + 1}, \cdots, x_i^{r + 1}) \subseteq (x_1, \cdots, x_i)^{r + 1}$. Also, recall that the Cohen structure theorem \cite[0C0S]{stacks-project} tells us that a regular Noetherian complete local $k$-algebra with residue field $k$ of dimension $n$ is isomorphic to $k[\![x_1, \cdots, x_n]\!]$.

    When $\mathrm{char\,}{k} = p > 0$, the $k$-algebra $A$ as above is a PD-thickening of $\Spec(k)$, so it is also possible to prove the lemma directly by crystalline formalism (e.g., some appropriate variant of \cite[07JH]{stacks-project}). 
\end{remark}


\paragraph{Acknowledgments} This collaboration started at a conference at the Simons-Laufer Mathematical Sciences
Institute (SLMath, previously known as MSRI), and the authors would like to thank SLMath for organizing
the event and their hospitality. The first author additionally thanks SLMath for hosting him during the spring
of 2023, and the Institut des Hautes Études Scientifiques (IHES) for hosting him and supporting his work
during the remainder of the year. The second author would like to thank Ananth Shankar and Mark Kisin
for their interest in the work. During the revision process of the work, the second author was supported
by the CUHK start up grant (Project No. 4937267), direct grant (Project No. 4053718) and an ECS grant
from the Hong Kong Research Grants Council (Project No. 24308225). We thank the anonymous referee
for carefully reviewing the paper and for many comments that improved its exposition.


\paragraph{Data Availability} Data sharing is not applicable to this article as no datasets were generated or analysed during the current study.

\paragraph{Conflict of Interest} On behalf of all authors, the corresponding author states that there is no conflict of interest. 

\printbibliography

\end{document}